\documentclass[11pt,a4paper]{article}
\usepackage{fancyhdr}
\usepackage[british]{babel}

\usepackage{amssymb, amsmath, amsthm, amsfonts, enumerate,bm}
\usepackage{amsmath,setspace,scalefnt}
\usepackage[usenames,dvipsnames,svgnames,table]{xcolor}
\usepackage{graphicx,tikz,caption,subcaption}
\usetikzlibrary{patterns}
\usetikzlibrary{shapes}
\usepackage{fullpage}
\usepackage{hyperref}
\usepackage{enumitem,verbatim}
\usepackage{multicol}
\usepackage{float}
\usepackage{cleveref}

\usepackage[margin=2.5cm]{geometry}

\definecolor{gray}{rgb}{0.25, 0.25, 0.25}

\usepackage{booktabs}  

\counterwithin{figure}{section}

\newtheorem{theorem}{Theorem}[section]
\newtheorem{lemma}[theorem]{Lemma}
\newtheorem{claim}[theorem]{Claim}
\newtheorem{cor}[theorem]{Corollary}

\newtheorem{conj}[theorem]{Conjecture}

\newtheorem*{observation*}{Observation}

\newtheorem{problem}[theorem]{Problem}

\newtheorem*{question*}{Question}
\newtheorem{innerdefinition}{Definition}
\newenvironment{definition*}
  {
   \innerdefinition}
  {\endinnerdefinition}

\theoremstyle{definition}

\theoremstyle{remark}

\usepackage{todonotes}

\newcommand{\mad}{{\mathrm{mad}}}

\title{Spectral extremal problems for non-bipartite graphs\\ without odd cycles}

\author{Lantao Zou,\quad 
Lihua Feng\thanks{
Corresponding authors. 
This paper was published on Discrete Mathematics 349 (2026) 114670. \\ 
Email addresses: \url{zoulantao123@163.com} (L. Zou), \url{fenglh@163.com} (L. Feng), 
\url{ytli0921@hnu.edu.cn} (Y. Li).},\quad
Yongtao Li$^*$ \\[3mm]
{\small School of Mathematics and Statistics, HNP-LAMA} \\ 
{\small Central South University, Changsha, Hunan, 410083, China}
}

\date{\today}

\begin{document}
\maketitle

\begin{abstract} 
A well-known result of Mantel asserts that every $n$-vertex triangle-free graph $G$ has at most $\lfloor n^2/4 \rfloor$ edges. Moreover, Erd\H{o}s proved that if $G$ is further non-bipartite, 
then $e(G)\le \lfloor {(n-1)^2}/{4}\rfloor +1$. 
Recently, Lin, Ning and Wu [Combin. Probab. Comput. 30 (2021)] established a spectral version by showing that 
if $G$ is a triangle-free non-bipartite graph on $n$ vertices, then $\lambda (G)\le \lambda (S_1(T_{n-1,2}))$, with equality if and only if  $G=S_1(T_{n-1,2})$, where $S_1(T_{n-1,2})$ is obtained from $T_{n-1,2}$ by subdividing an edge. 
In this paper, we investigate the maximum spectral radius of  a non-bipartite graph without some short odd cycles. Let $C_{2\ell +1}(T_{n-2\ell, 2})$ be the graph obtained by identifying a vertex of $C_{2\ell+1}$ and a vertex of the smaller partite set of $T_{n-2\ell ,2}$.  
We prove that for $1\le \ell < k$ and $n\ge 187k$, 
if $G$ is an $n$-vertex $\{C_3,\ldots ,C_{2\ell -1},C_{2k+1}\}$-free non-bipartite graph, 
then $\lambda (G)\le \lambda (C_{2\ell +1}(T_{n-2\ell, 2}))$, with equality if and only if $G=C_{2\ell +1}(T_{n-2\ell, 2})$. 
This result could be viewed as a spectral analogue of a min-degree result due to Yuan and Peng [European J. Combin. 127 (2025)]. Moreover,  our result extends a result of Guo, Lin and Zhao [Linear Algebra Appl. 627 (2021)] as well as a recent result of Zhang and Zhao [Discrete Math. 346 (2023)] since we can get rid of the condition that $n$ is sufficiently large. 
The argument in our proof is quite different and makes use of the classical spectral stability method and the double-eigenvector technique. The main innovation lies in a more clever argument that guarantees a subgraph to be bipartite after removing few vertices, which may be of independent interest. 
\end{abstract}



\section{Introduction} 

 Let $\mathcal{F}$ be a family of simple graphs. A graph ${G}$ is called $\mathcal{F}$-free if it does not contain any member of $\mathcal{F}$ as a subgraph. 
The Tur\'{a}n number of $\mathcal{F}$ is defined as the maximum number of edges in an  $n$-vertex $\mathcal{F}$-free graph and is denoted by $\mathrm{ex}(n,\mathcal{F})$. 
First of all, we review some fundamental results for triangle-free graphs. 
Let $T_{n,2}$ be the $n$-vertex bipartite Tur\'{a}n graph, which is a complete bipartite graph whose two parts have sizes as equal as possible. The classical Mantel  theorem (see, e.g., \cite{Bollobas78}) states that 
if $G$ is an $n$-vertex triangle-free graph, then $e(G)\le \lfloor {n^2}/{4}\rfloor$, with equality if and only if 
$G=T_{n,2}$. 
There are various extensions and generalizations on Mantel's theorem in the literature; 
see, e.g., \cite{BT1981,Bon1983}. 
The classical stability shows that if $G$ is an $n$-vertex triangle-free graph with $m= \lfloor {n^2}/{4}\rfloor -q$ edges, where $q\ge 0$, then 
 $G$ can be made bipartite by removing at most $q$ edges; see \cite{Fur2015}.  
Moreover, Erd\H{o}s (see \cite[p. 306]{BM2008}) presented a stability result by showing that if 
$G$ is triangle-free and  
 $e(G)> \lfloor {(n-1)^2}/{4}\rfloor +1$, 
 then $G$ is bipartite. 
 This threshold can be achieved by many unbalanced blow-ups of $C_5$. 
In addition, a variant was studied in terms of the minimum degree by 
Andr\'{a}sfai, Erd\H{o}s and S\'{o}s \cite{AES1974}, who proved that if $G$ is a triangle-free graph with $\delta (G)> {2n}/{5}$, 
then $G$ is bipartite. Moreover, the factor ${2}/{5}$ is the best possible as witnessed by the balanced blow-up of $C_5$. 
We refer the readers to \cite{Bra2003,Thom2002,Luc2006,Jin1995,GL2011} for related stability results involving the cliques and \cite{CJ2002,Thom2007} for short odd cycles.

\medskip 

In this paper, we focus on various stability results on graphs that forbid odd cycles. 

\subsection{Min-degree stability for odd cycles}

The degree stability for short odd cycles was  studied Andr\'{a}sfai, Erd\H{o}s and S\'{o}s \cite{AES1974}. 

\begin{theorem}[Andr\'{a}sfai--Erd\H{o}s--S\'{o}s \cite{AES1974}] \label{thm-AES}
Let $k\ge 2$ and $n\ge 2k+3$ be integers. 
If $G$ is a $\{C_3,C_5,\ldots ,C_{2k+1}\}$-free graph on $n$ vertices with   
\[ \delta (G) \ge \frac{2}{2k+3}n,\] 
then $G$ is bipartite, unless $G$ is the balanced blow-up of $C_{2k+3}$. 
\end{theorem} 

Focusing on forbidding a single odd cycle, 
H\"{a}ggkvist \cite{Hagg1982} showed that for each $k\in \{2,3,4\}$, if $G$ is an $n$-vertex $C_{2k+1}$-free graph with $\delta (G)> \frac{2}{2k+3}n$, then $G$ is bipartite. 
In addition, H\"{a}ggkvist \cite{Hagg1982} remarked that 
this result cannot be extended to $k\ge 5$. 
In 2024, Yuan and Peng \cite{YP2024} 
established the min-degree stability 
for a single odd cycle $C_{2k+1}$ with $k\ge 5$.  

\begin{theorem}[Yuan--Peng \cite{YP2024}] 
\label{thm-YP2024}
Let $k\ge 5$ and $n\ge 21000k$ be integers. 
If $G$ is a $C_{2k+1}$-free graph on $n$ vertices with 
\[  \delta (G) \ge \frac{n}{6}, \] 
then $G$ is bipartite, unless  $G$ is the 
H\"{a}ggkvist construction: take three vertex-disjoint copies of 
$K_{\frac{n}{6},\frac{n}{6}}$, select a vertex from each of them 
and put a triangle in the three vertices.  
\end{theorem}

Observe that the extremal graph constructed  in \Cref{thm-YP2024} contains a triangle. 
It is natural to consider the second optimal extremal graph further by forbidding a triangle. More generally, Yuan and Peng \cite{YP2023}  proved the further stability for $ C_{2k+1}$-free
graphs by excluding all short odd cycles with lengths at most $2\ell -1$, where $ \ell \le k$ is a positive integer. 

\begin{theorem}[Yuan--Peng \cite{YP2023}]\label{thm-YP-many-cycles}
Let $2\le \ell\le k, n\ge 1000k^{8}$ be integers and let $G$ be a $\{C_3,C_5,\dots , C_{2\ell-1},C_{2k+1}\}$-free graph on $n$ vertices.
\begin{itemize}

\item[\rm (i)]
If $ \ell>\frac{2k-1}{8}$ and $\delta(G)\ge \frac{2}{2k+3}n$, then $G$ is bipartite, or $G$ is a balanced blow-up of $C_{2k+3}$.

\item[\rm (ii)]
If $2 \le \ell<\frac{2k-1}{8}$ and $\delta(G)\ge \frac{n}{2(2\ell+1)}$, then $G$ is bipartite, or $G$ is a graph obtained by taking $2\ell+1$ vertex-disjoint copies of $K_{\frac{n}{2(2\ell+1)},\frac{n}{2(2\ell+1)}}$ and selecting a vertex in each of them such that these vertices form a cycle of length $2\ell+1$.
\end{itemize}
\end{theorem}

\subsection{Spectral stability for odd cycles}

In recent years, there has been significant growth in the field of spectral extremal graph theory, which is an advanced and useful area of extremal combinatorics that combines graph theory with algebraic methods, particularly through the use of eigenvalues of associated matrices to study the combinatorial structures of graphs. This field has gained great popularity in the last few years due to its potential applications in various domains such as theoretical computer science, combinatorial geometry, and,  more broadly, in the study of complex networks. 

The {\it adjacency matrix} of a graph $G$ is defined as $A(G)=(a_{ij})_{n\times n}$ with $a_{ij}=1$ if $ij\in E(G)$, and $a_{ij}=0$ otherwise. 
The {\it spectral radius} of $G$, denoted by $\lambda (G)$, is the maximum modulus of eigenvalues of $A(G)$. 
The well-known spectral Mantel theorem \cite{Niki2007laa2} asserts that if $G$ is a triangle-free graph on $n$ vertices, then $\lambda (G)\le \lambda (T_{n,2})$, and the equality holds if and only if $G=T_{n,2}$. 
In 2021, 
Lin, Ning and Wu \cite{LNW2021} 
determined the maximum spectral radius of non-bipartite  triangle-free graphs. 
Let $S_1(T_{n-1,2})$ be obtained from 
$T_{n-1,2}$ by subdividing an edge.
It was shown in \cite{LNW2021} that
if $G$ is a triangle-free graph with  
$ \lambda(G)\geq \lambda(S_1(T_{n-1,2}))$, 
    then $G$ is bipartite, unless $G=  S_1(T_{n-1,2})$. 
This result initiated the study on spectral  problems for non-bipartite graphs.

Lin, Ning and Wu \cite{LNW2021} (see \cite{CJ2002} independently) extended  the problem by forbidding short odd cycles and proved that if 
$G$ is a $\{C_3,C_5,\ldots,C_{2k+1}\}$-free graph on $n\ge 2k+3$ vertices with 
\begin{equation} \label{eq-3-5-2k+1} 
e(G) > \left\lfloor \frac{(n-2k+1)^2}{4} \right\rfloor  +2k-1, 
\end{equation}
then $G$ is bipartite. 
Furthermore, they proposed a more
general question to establish the corresponding spectral version of (\ref{eq-3-5-2k+1}).  
Soon after, this question was solved by 
Lin and Guo \cite{LG2021} and 
Li, Sun and Yu \cite{LSY2022} independently.  
Let $S_{2k-1}(T_{n-2k+1,2})$ be the graph obtained by replacing an edge of the bipartite Tur\'{a}n graph $T_{n-2k+1,2}$ with a path $P_{2k+1}$ of order $2k+1$; see Figure \ref{Fig-Extremal}.

\begin{theorem}[Lin--Guo \cite{LG2021}; Li--Sun--Yu \cite{LSY2022}]
\label{thm-LG-LSY}
    Let $k\ge 2$ and $n\ge 2k+9$ be integers. If $G$ is a $\{C_3,C_5,\ldots,C_{2k+1}\}$-free graph on $n$ vertices with  
    $$
    \lambda(G)\geq \lambda(S_{2k-1}(T_{n-2k+1,2})),
    $$
    then $G$ is bipartite, unless $G=  S_{2k-1}(T_{n-2k+1,2})$.
\end{theorem}

Nikiforov \cite{Niki2008} studied the spectral extremal problem for $C_{2k+1}$-free graphs and 
proved that if $k\ge 2$ and $n\ge 320(2k+1)$ is sufficiently large, 
and $G$ is a $C_{2k+1}$-free graph of order $n$, then 
$\lambda (G)\le \lambda (T_{n,2})$. 
We refer to \cite{ZL2022jgt,LN2023,Zhang-wq-2024} for recent generalizations.  
 Correspondingly, 
it is interesting to establish a spectral stability for Nikiforov's result, i.e., weakening the condition in Theorem \ref{thm-LG-LSY} by forbidding a single odd cycle $C_{2k+1}$. 
The case $k=1$ reduces to the aforementioned result of Lin, Ning and Wu \cite{LNW2021}; 
the case $k=2$ was proved by Guo, Lin and Zhao \cite{GLZ2021} for $n\ge 21$; the general case for $k\ge 3$ was recently studied by Zhang and Zhao \cite{ZZ2023} for sufficiently large $n$. 
To begin with, 
we denote by $C_{2\ell+1}(T_{n-2\ell ,2})$ the graph obtained by identifying a vertex of $C_{2\ell +1}$ and a vertex 
of $T_{n-2\ell ,2}$ belonging to the vertex part with {\bf smaller} size; 
see Figure \ref{Fig-Extremal}. 

 \begin{figure}[H]
\centering
\includegraphics[scale=0.8]{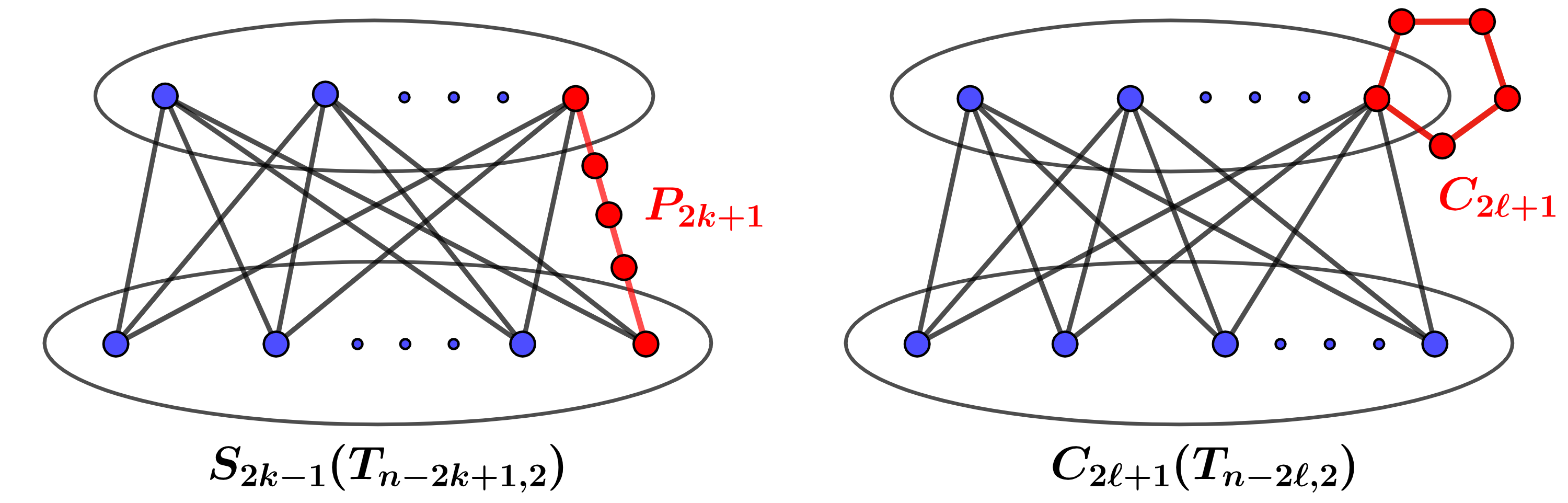} 
\caption{The extremal graphs in Theorem \ref{thm-odd-bi-stable}.}
\label{Fig-Extremal}
\end{figure}

\begin{theorem}[Guo--Lin--Zhao  \cite{GLZ2021}; Zhang--Zhao \cite{ZZ2023}]
\label{thm-non-bi-C2k+1} 
     Let $k\geq 2$ and $n$ be  sufficiently large with respect to $k$ ($n\ge 21$ if $k=2$). 
If $G$ is a $C_{2k+1}$-free graph on $n$ vertices with  
     $$
     \lambda(G)\geq \lambda(C_{3}(T_{n-2,2})),
     $$
     then $G$ is bipartite, unless $G=  C_{3}(T_{n-2,2})$.
\end{theorem}

 Li, Sun and Wei \cite{LSW2024} 
and Fang and Lin \cite{FL2024+}
extended this result to $\theta_{1,q,r}$-free graphs for even $q$, where $\theta_{1,q,r}$ is the graph obtained from internally vertex-disjoint paths of lengths $1,q,r$ by sharing a common pair of endpoints.  Clearly, we see that $C_{q+1}$ is a proper subgraph of $\theta_{1,q,r}$.

\section{Main results}

In this paper, we prove the following result, which extends Theorem \ref{thm-non-bi-C2k+1}.

\begin{theorem}[Main result] 
\label{thm-odd-bi-stable}
    Let $1\leq \ell\leq k$ and $G$ be a  $\{C_3,C_5,\ldots, C_{2\ell-1},C_{2k+1}\}$-free graph on $n$ vertices, where $n\geq 187k$. 
    \begin{itemize}
        \item[\rm (a)] If $\ell=k$ and $\lambda(G)\geq \lambda(S_{2k-1}(T_{n-2k +1,2}))$, then $G$ is bipartite, unless $G=  S_{2k-1}(T_{n-2k +1,2})$. 
        
        \item[\rm (b)] If $\ell\leq k-1$ and $\lambda(G)\geq \lambda(C_{2\ell+1}(T_{n-2\ell ,2}))$, then $G$ is bipartite, unless $G=  C_{2\ell+1}(T_{n-2\ell ,2})$.
    \end{itemize}
\end{theorem}

Although we say that Theorem \ref{thm-odd-bi-stable} is a  spectral analogue of Theorem \ref{thm-YP-many-cycles}, 
there are some significant differences between them. 
The min-degree stability results of Theorems \ref{thm-AES} -- \ref{thm-YP-many-cycles} are more difficult because of the requirement that every vertex must have large degree. 
To some extent, the spectral radius is analogous to the average degree. Moreover, the extremal graphs for these two problems are completely different. The extremal graph for the spectral problem is a large complete bipartite graph $T_{n-2\ell ,2}$ with a small odd cycle $C_{2\ell +1}$ attached, where the average degree is approximately \(\frac{n}{2}\). 
In contrast, the extremal graph for the min-degree problem is similar to H\"{a}ggkvist's construction, i.e., attaching a   complete bipartite graph $K_{\frac{n}{2(2\ell+1)},\frac{n}{2(2\ell+1)}}$ to every vertex of the odd cycle $C_{2\ell +1}$, where the min-degree is exactly \(\frac{n}{2(2\ell +1)}\).

We would like to emphasize that Theorem \ref{thm-odd-bi-stable} requires merely a linear bound on $n$, namely, $n\ge 187k$. This substantially improves the polynomial bound in \Cref{thm-YP-many-cycles} as well as the non-specified large bound in \Cref{thm-non-bi-C2k+1}. 
We point out here that we made no effort to optimize the coefficient, and rather focused on making the argument as general and easy to read as possible. By being significantly more careful and tweaking multiple parts of our argument, we can improve this constant.   
 Finding specified bounds (even linear or tight bounds), instead of sufficiently large bounds, turns out to be interesting and significant for some extremal graph problems in the literature; 
 see, e.g., \cite{LN2021outplanar,ZL2022jgt,LP-ejc2022,LLP2022,LFP2024-triangular} for recent spectral graph results.

In view of Theorems \ref{thm-LG-LSY} and \ref{thm-non-bi-C2k+1}, we know that part (a) in Theorem \ref{thm-odd-bi-stable} holds for $n\ge 2k+9$, and part (b) holds in the case $\ell =1$ for sufficiently large $n$. 
In our proof, we shall show that 
Theorem \ref{thm-odd-bi-stable} holds for
$ 1 \leq \ell\leq k-1$. 
 Here, we include a proof in  the case $\ell =1$ again, since our proof drops the condition in Theorem \ref{thm-non-bi-C2k+1} that requires the order $n$ to be sufficiently large. 

 \paragraph{Our approach.} 
 Our proof of part (b) in Theorem \ref{thm-odd-bi-stable} is quite different from that of \Cref{thm-non-bi-C2k+1} in  \cite{GLZ2021,ZZ2023}. 
 Our argument yields the structural results 
 that characterize the almost-extremal graphs. 
 To some extent, our approach develops the classical spectral stability method in \cite{CFTZ20}.  
 By applying a Brauldi--Hoffman--Tur\'{a}n type result (\Cref{lem-ZLS}), 
 we can show that the spectral extremal graph $G$ 
 must have a large number of edges. 
 More precisely, a non-bipartite $\{C_3,C_5,\ldots ,C_{2\ell -1},C_{2k+1}\}$-free graph $G$ with the maximum spectral radius must have at least $n^2/4 -kn +k$ edges, and $G$ contains very few ``bad'' vertices with small degree; see Lemmas \ref{lem-graph-size} and \ref{lem-bound-L}. 
As is well-known, the conventional stability method allows us to remove at most $o(n^2)$ edges from $G$ to make it bipartite. 
We point out that the supersaturation-stability method developed Li, Feng and Peng \cite{LFP2024-triangular,LFP-count-bowtie} reveals that it is sufficient to remove less than $2kn$ edges since every $C_{2k+1}$-free graph contains less than $\frac{3}{2}kn^2$ triangles. 
Fortunately, in our setting, we can improve the procedure of the stability method by removing less than $\sqrt{kn}$ bad vertices only, rather than removing a number of edges. 
This is a key contribution of our paper.

 The main innovation of our proof lies in
utilizing a result on the minimum degree for 
 weakly pancyclic graphs (\Cref{weakly-pancyclic}), which helps us to show that after excluding few bad vertices, the resulting subgraph is (complete) bipartite; see \Cref{V_1V_2}. 
 This method turns out to be more efficient than the conventional stability method, and it 
 is targeted at the extremal problems involving cycles.  
 Under this structural result, we prove further that all bad vertices are contained in a shortest odd cycle, say $C$, and  
 all but at most one vertex of $C$ have no neighbors in $G-C$; see Lemmas \ref{lem-L-2k}, \ref{lem-pathlength} and 
\ref{lem-deg-zero}. 
Finally, using the so-called double-eigenvector technique, we will show that 
excepting the vertices of the shortest odd cycle, the remaining vertices form a balanced complete bipartite graph; see \Cref{lem-balanced}.  
In a word, we must have $G=C_{2\ell +1}(T_{n-2\ell ,2})$, which is the expected extremal graph.

\subsection{Cycles with consecutive lengths}

A well-known result of Bondy \cite{Bon1971jctb} states that if $G$ is an $n$-vertex Hamiltonian graph with $m> \lfloor n^2/4 \rfloor $ edges, then 
$G$ contains all cycles $C_{\ell}$ with $\ell \in [3,n]$. 
Afterwards, Bondy \cite{Bon1971dm} and Woodall \cite{Woo1972} proved that if $G$ is an $n$-vertex graph with $m> \lfloor n^2/4\rfloor $ edges, then $G$ contains a cycle $C_{\ell}$ for every $3\le \ell \le \lfloor \frac{n+3}{2}\rfloor$. 
In 2008, Nikiforov \cite{Niki2008} studied the spectral condition 
on the existence of cycles with consecutive lengths and showed that 
if $n$ is sufficiently large and 
$G$ is an $n$-vertex graph with $\lambda (G)>  \lambda (T_{n,2})$, 
 then $G$ contains a cycle $C_{\ell}$ 
for every  $ \ell \in [3, n/320]$. 
Moreover, Nikiforov remarked that the constant $1/320$ can be increased with careful calculations. 
Hence, He proposed the following problem.

\begin{problem}[Nikiforov \cite{Niki2008}] \label{quesniki}
What is the maximum $c$ such that for  sufficiently large $n$, every graph $G$ of order $n$ with 
$\lambda (G)> \lambda (T_{n,2})$ 
contains a cycle $C_{\ell}$ for every $\ell \le (c+o(1))n$. 
\end{problem}

Different from the conventional density version, 
Nikiforov constructed an example that shows 
$c\le \frac{3-\sqrt{5}}{2}\approx 0.382$ in Problem \ref{quesniki}.   
The recent progress on this problem is listed below: 
Ning and
Peng \cite{NP2020} slightly improved $c\ge 1/160$;  
 Zhai and 
Lin  \cite{ZL2022jgt} showed the result to $c\ge 1/7$; 
Li and Ning \cite{LN2023} proved that $c \ge 1/4$, and Zhang \cite{Zhang-wq-2024} 
improved that $c\ge 1/3$.

In the sequel, we investigate Problem \ref{quesniki} for non-bipartite graphs. 
In particular, the case $\ell =1$ in Theorem \ref{thm-odd-bi-stable} implies that if $n\ge 187k$ and $G$ is 
a non-bipartite graph of order $n$ with $\lambda (G)\ge \lambda (C_3(T_{n-2,2}))$, then 
$G$ contains a $C_{2k+1}$, unless 
$G=  C_3(T_{n-2,2})$. 
As a quick application, 
we obtain the following corollary 
involving consecutive odd cycles in non-bipartite graphs.

\begin{cor} \label{cor-conse-odd-cycles}
If $G$ is a non-bipartite graph of order $n$ with $\lambda (G)\ge \lambda (C_3(T_{n-2,2}))$, then $G$ contains an odd cycle $C_{2k+1}$ 
    for each $k \in [2,n/187]$, unless $G=  C_3(T_{n-2,2})$. 
\end{cor}

This corollary allows us to put $k$ in a wide range that can grow linearly with $n$. Hence, we substantially 
improve the bound obtained by Zhang and Zhao \cite[Theorem 2]{ZZ2023}. 
We remark here that the triangle cannot be guaranteed in such a graph $G$ in \Cref{cor-conse-odd-cycles}, since $S_1(T_{n-1,2})$ is a non-bipartite graph and  $\lambda (S_1(T_{n-1,2}))> 
\lambda (C_3(T_{n-2,2}))$, but it does not contain triangle.  

Motivated by Nikiforov's problem, one may ask: does the graph $G$ in Corollary \ref{cor-conse-odd-cycles} have even cycles with consecutive lengths as well? In this paper, we answer this question as follows.

\begin{theorem}\label{thm-non-bi-cycles}
    For each $\varepsilon \in (0,\frac{1}{3})$, there exists  $N(\varepsilon)$ such that if $G$ is a non-bipartite graph on $n \ge N(\varepsilon)$ vertices and $\lambda(G)>\lambda(C_3(T_{n-2,2}))$, then $G$ has a cycle $C_{\ell}$ for every $\ell \in [4, (\frac{1}{3} - \varepsilon)n]$.
\end{theorem}

\Cref{thm-non-bi-cycles} not only guarantees the existence of even cycles with consecutive lengths, but also improves the range of $k$ in Corollary \ref{cor-conse-odd-cycles} for large $n$. Our proof of \Cref{thm-non-bi-cycles} is different from that of Theorem \ref{thm-odd-bi-stable}. Our argument adopts a line  similar to  that in \cite{LN2023,Zhang-wq-2024}. This is not the main ingredient of this paper, so we postpone the detailed argument to the Appendix \ref{App-B}.

\section{Preliminaries} 

\paragraph{Notation.} 
Let $G= (V( G), E( G))$ be a simple graph with vertex set $V( G)$ and edge set $E(G)$. 
We denote by $e( G)$ the number of edges in $G$. 
For a vertex
$v\in V(G)$, let $N(v)$ be the set of neighbors of $v$ in $G$. The {\it degree} $d(v)$ of $v$ is equal to $|N(v)|$. 
The minimum and maximum degrees are denoted by $\delta(G)$ and $\Delta(G)$, respectively. For $V_1, V_2\subseteq V( G)$, let $E( V_1, V_2)$ denote the set of edges of $G$ between $V_1$ and $V_2$, and let $e(V_1,V_2)=|E(V_1,V_2)|$. For every  $S\subseteq V(G)$, we write $N(S)=\cup_{u\in S}N(u)$ and  $d_S(v)=|N_S(v)|=|N(v)\cap S|$. Denote by $G - S$ the subgraph obtained from $G$ by deleting all vertices in $S$ with their incident edges. Let $G[S]$ be  the subgraph induced by $S$, i.e., 
a graph whose vertex set is $S$ and the edge set consists of all the edges of $G$ with two  ends in $S$.  
Since $A(G)$ is a non-negative matrix, 
the Perron--Frobenius theorem implies that the spectral radius $\lambda (G)$ is actually a largest eigenvalue of $A(G)$ and there exists a 
non-negative eigenvector of $A(G)$ corresponding to $\lambda(G)$. 
Throughout the paper, 
let $\bm{x} = ( x_1, \ldots , x_n)^{T}$ be such a non-negative eigenvector. 
The following two facts will be frequently used: The first states that for every $i\in V(G)$, we have 
$$
\lambda(G)x_i=\sum_{j\in N(i)}x_j.
$$ 
The second fact concerns the Rayleigh quotient:
$$\lambda(G)=\max_{\bm{x}\in\mathbb{R}^n}\frac{\bm{x}^{T}A(G)\bm{x}}{\bm{x}^{T}\bm{x}} 
=\max_{\bm{x}\in\mathbb{R}^n}\frac{1}{\bm{x}^{T}\bm{x}} 2\sum_{ij\in E(G)}x_ix_j,$$
where we write $\sum_{ij\in E(G)}$ for the sum over each edge in $E(G)$ once.

\medskip 
In our proof of \Cref{thm-odd-bi-stable}, we need to use the following lemmas. 

\begin{lemma}[See \cite{Bon1971dm,Woo1972}]  \label{lem-Bon-Woo} 
Let $G$ be a  graph  on $n$ vertices 
with  
\[  e(G)\ge  \left\lfloor \frac{n^2}{4} \right\rfloor +1. \]   
Then $G$ contains a cycle  $C_{\ell}$ 
for each $3\le \ell \le \lfloor \frac{n+3}{2} \rfloor$. 
\end{lemma}

 \begin{lemma}[See \cite{Nosal1970,Niki2002cpc,Niki2006-walks}] \label{lem-Nosal} 
Let $G$ be a triangle-free  graph with $m$ edges. Then 
\[ {  \lambda (G)\le \sqrt{m} }, \] 
with equality if and only if 
$G$ is a complete bipartite graph (possibly with some isolated vertices).
\end{lemma}


 \begin{lemma}[See \cite{ZLS2021}] 
 \label{lem-ZLS}
If $G$ is a graph with $m$ edges and 
\[  \lambda (G) > \frac{k-1/2 + \sqrt{4m + (k-1/2)^2}}{2} , \]
then $G$ contains a cycle of length $t$ 
for every $t\le 2k+2$. 
\end{lemma}

The following lemma \cite{WXH2005} gives an operation that increases the spectral radius. 

 \begin{lemma}[See \cite{WXH2005}]  \label{lem-WXH}
 Assume that $G$ is a connected graph with 
 $u,v\in V$ and $w_1,\ldots ,w_s\in N(v)\setminus N(u)$. 
 Let $\bm{x}=(x_1,\ldots ,x_n)^{\top}$ be the Perron vector 
 with $x_v$ corresponding to the vertex $v$. 
 Let $G'=G-\{vw_i: 1\le i\le s\} + 
 \{uw_i: 1\le i \le s\}$. If ${x}_u\ge {x}_v$, 
 then $\lambda (G') > \lambda (G)$. 
 \end{lemma}

A walk $v_1v_2\cdots v_k$ in a graph $G$ is called an {\it internal path} if these $k$ vertices are distinct (except possibly $v_1=v_k$), $d_G(v_1)\ge 3$, $d_G(v_k)\ge 3$ and 
$d_G(v_2)= \cdots =d_G(v_{k-1})=2$ unless $k=2$. 
 Let $Y_n$ be the graph obtained from an induced path $v_1v_2\cdots v_{n-4}$ by attaching two pendant vertices to $v_1$ and other two pendant vertices to $v_{n-4}$. 
 Let $G_{uv}$ denote the graph obtained from $G$ by subdividing the edge $uv$, i.e., adding a new vertex on the edge $uv$. 

Hoffman and Smith \cite{HS1975} proved the following result (see, e.g., \cite[p. 79]{CDS1980}). 

\begin{lemma}[See \cite{HS1975}] \label{lem-HS}
    Let $G$ be a connected graph with $uv\in E(G)$. If $uv$ belongs to an internal path of $G$ and $G\not=  Y_n$, then $\lambda (G_{uv})< \lambda (G)$. 
\end{lemma}

\begin{lemma}
    \label{lem-final}
    If $\ell < t$ are positive integers, then 
    \[ \lambda (C_{2t+1}(T_{n-2t,2})) < \lambda (C_{2\ell +1}(T_{n-2\ell ,2})). \]  
\end{lemma}

\begin{proof} 
{Note that $C_{2t+1}(T_{n-2t,2})$ can be obtained from $C_{2\ell +1}(T_{n-2t ,2})$ by subdividing an edge of the odd cycle $C_{2\ell +1}$ $2(t- \ell)$ times.} Then 
    applying \Cref{lem-HS} yields 
    $\lambda (C_{2t+1}(T_{n-2t,2})) < 
    \lambda (C_{2\ell +1}(T_{n-2t,2})) < 
    \lambda (C_{2\ell+1}(T_{n-2\ell,2}))$, where the last inequality holds 
   since  
    $C_{2\ell+1}(T_{n-2t,2})$ is a proper subgraph of $C_{2\ell+1}(T_{n-2\ell,2})$, as desired.  
\end{proof}

Let $g(G)$ and 
$c(G)$ be the lengths of a shortest and longest cycle in $G$, respectively.  

\begin{lemma}[See \cite{BFG1998}]
\label{weakly-pancyclic}
 If $G$ is a non-bipartite graph with order $n$ and 
minimum degree $\delta (G)\ge \frac{n+2}{3}$, 
 {then $G$ contains  a cycle of every length between $g(G)$ and $c(G)$ with $g(G)\in\{3,4\}$.} 
\end{lemma}

\begin{lemma}[See \cite{EG1959}] \label{Erdos-Gallai}
If $G$ is an $n$-vertex graph with more than 
$\frac{1}{2}t(n-1)$ edges, then $G$ contains a cycle with length at least $t+1$, i.e., $c(G)\ge t+1$. 
\end{lemma}

\section{Proof of Theorem \ref{thm-odd-bi-stable}}
In this section, {we first give an approximate structure description of the almost-extremal graphs (see Lemmas \ref{lem-bound-L} -- \ref{V_1V_2}).} 
For notational convenience, we denote 
\[ \mathcal{C}_{\ell,k}:=\{C_3,\ldots,C_{2\ell-1},C_{2k+1}\}.\]  
Suppose that $G$ attains the maximum spectral radius among all $n$-vertex  $\mathcal{C}_{\ell,k}$-free non-bipartite graphs. Then $\lambda (G)\ge \lambda (C_{2\ell +1}(T_{n-2\ell ,2}))$. 
Our goal is to show that $G=C_{2\ell +1}(T_{n-2\ell ,2})$. 
Clearly, $G$ is connected. So $\bm{x} \in \mathbb{R}^n$ is a positive eigenvector of $\lambda(G)$. 
By scaling, we may assume that $\max \{x_i:i\in V(G)\}=1$ and $z\in V(G)$ is a vertex with $x_z=1$. If there are several such vertices, then we choose any one of them. 
Since $G$ is non-bipartite, we know that $G$ contains an odd cycle. Let $C=u_1\ldots u_{2t+1}u_1$ be a shortest odd cycle in $G$. Next, we prove that $\ell\leq t\leq k-1$.

\begin{lemma}\label{lem-cycle-length}
   We have $\ell\leq t\leq k-1$. 
\end{lemma}

\begin{proof}
{ 
Clearly, we have $\ell \le t$ and $t\neq k$. 
    If $t\ge k+1$, then $G$ is $\mathcal{C}_{\ell ,k}$-free. Theorem \ref{thm-LG-LSY} gives 
    $$\lambda(G)\leq \lambda(S_{2k-1}(T_{n-2k+1,2})) < \Delta(S_{2k-1}(T_{n-2k+1,2}))
        \leq \left\lceil\frac{n-2k+1}{2} \right\rceil.$$
    Moreover, we have $\lambda(C_{2\ell+1}(T_{n-2\ell ,2}))>\lfloor\frac{n-2\ell}{2}\rfloor$. Since $\ell\leq k-1$, we get 
    \begin{align*}
        \lambda(G)\leq \left\lceil\frac{n-2k+1}{2} \right\rceil 
        \leq \left\lfloor\frac{n-2\ell}{2} \right\rfloor 
         <\lambda(C_{2\ell+1}(T_{n-2\ell ,2})),
    \end{align*}
     contradicting with the assumption. This completes the proof.} 
\end{proof}

\begin{lemma}\label{lem-graph-size}
  We have $\lambda (G) > \lfloor \frac{n-2\ell}{2}\rfloor > \frac{n}{2} -k$ and 
  $e(G)\geq \frac{n^2}{4}-kn+k$.
\end{lemma}

\begin{proof}
We know from the assumption that 
\begin{equation*}
    \label{eq-lower-bound}
    \lambda (G)\geq  \lambda (C_{2\ell +1}(T_{n-2\ell ,2})) > \left\lfloor \frac{n-2\ell}{2}\right\rfloor > \frac{n}{2}-k. 
\end{equation*} 
    In the case $\ell \ge 2$, the assumption implies that $G$ is triangle-free. So Lemma \ref{lem-Nosal} gives $e(G)\ge \lambda^2(G) \ge \frac{n^2}{4} - kn + k^2$.  
    Next, we consider the case $\ell=1$. 
    Then $G$ forbids a single odd cycle $C_{2k+1}$ and $\lambda (G)\ge \lambda (C_3(T_{n-2,2}))> \lfloor \frac{n-2}{2} \rfloor$.
    We obtain from 
    Lemma \ref{lem-ZLS} that 
    $$\left\lfloor \frac{n-2}{2} \right\rfloor < \lambda(G)< \frac{k- 1+\sqrt{4m+(k- 1)^2}}{2}.$$
    It follows that 
    $$m > \frac{n^2}{4}-\frac{1}{2}(k+2)n + \frac{3(2k+1)}{4}\geq \frac{n^2}{4}-kn+k. 
    \qedhere $$ 
\end{proof}

We remark that there is another way to get a lower bound on $e(G)$ by a well-known inequality $e(G)\ge \lambda^2(G)- \frac{3t(G)}{\lambda (G)}$, where $t(G)$ denotes the number of triangles in $G$. We can see that 
every $C_{2k+1}$-free graph $G$ satisfies $t(G)=\frac{1}{3}\sum_{v\in V} e(N(v))  \le \frac{1}{3}\sum_{v\in V} kd(v) \le \frac{2}{3}km \le \frac{1}{6}kn^2$. 

\begin{lemma}\label{lem-bound-L}
   Let $L:=\{v\in V(G): d(v)\leq\frac{n}{2} -\frac{5}{4}\sqrt{kn}\}$. Then $|L| \le  \sqrt{kn}$.
\end{lemma}

\begin{proof}
    Suppose on the contrary that $|L|> \sqrt{kn}$.  Then there exists a subset $S \subseteq L $ such that $|S|=\lceil \sqrt{kn}\rceil$. By Lemma \ref{lem-graph-size}, we have
    \begin{align*}
e(G- S)& \geq e(G)-\sum_{v\in S}d_{G}(v)  \\
&> \frac{n^2}{4} - kn+k-\sqrt{kn}
\left(\frac{n}{2} -\frac{5}{4}\sqrt{kn} \right) \\
&>\frac{(n-\lceil\sqrt{kn}\rceil)^2}{4}.
\end{align*}
Since $n\ge 9k$, we have $|V(G-S)|=n-\lceil\sqrt{kn}\rceil>4k$. We know from Lemma \ref{lem-Bon-Woo} that the induced subgraph 
$G- S$ contains a copy of $C_{2k+1}$, a contradiction.
\end{proof}

From Lemma \ref{lem-bound-L}, we know that we can remove at most $\sqrt{kn}$ vertices from $G$ to obtain a large $C_{2k+1}$-free subgraph $G-L$ with large minimum degree. From Theorem \ref{thm-YP2024}, we can see that $G-L$ is bipartite under a slightly stronger condition $n\ge 21000k$.  
In what follows, we weaken this bound and apply Lemmas \ref{weakly-pancyclic} and \ref{Erdos-Gallai} to show that the subgraph $G-L$ is bipartite for every $n\ge 187k$. Hence, we obtain an approximate structure of the spectral extremal graph.   

\begin{lemma}\label{V_1V_2} 
Each vertex of $G-L$ has at least $\frac{1}{2}n-\frac{9}{4}\sqrt{kn}$ neighbors, and $G-L$ is a bipartite graph. {Assume that $V(G-L)=V_1\sqcup V_2$, where $V_1$ and $V_2$ are partite sets of $G-L$}. Then 
\[ \frac{1}{2}n-\frac{9}{4}\sqrt{kn}\leq|V_1|,|V_2|\leq \frac{1}{2}n+\frac{9}{4}\sqrt{kn}. \] 
\end{lemma}

\begin{proof}
    By definition, we have $d(v)> \frac{1}{2}n-\frac{5}{4}\sqrt{kn}$ for every $v\in V(G-L)$. Since $|L|\leq \sqrt{kn}$, we get $d_{G-L}(v)> \frac{1}{2}n-\frac{9}{4}\sqrt{kn}$. Since $n\geq 187k$, it follows that $\delta(G-L)>\frac{1}{2}n-\frac{9}{4}\sqrt{kn}\geq \frac{n+2}{3}$.

    We claim that $G-L$ is bipartite. 
    Assume on the contrary that $G-L$ is non-bipartite. 
    Then by \Cref{Erdos-Gallai},
    we know that $G-L$ contains a cycle of length at least  $\delta(G-L)+1> \frac{1}{2}n-\frac{9}{4}\sqrt{kn}\geq 2k+1$ since $n\geq 40k$. 
    The weakly pancyclicity in \Cref{weakly-pancyclic} implies that
 $G-L$ contains all cycles with lengths  between $4$ and $2k+1$, which leads to a contradiction. Thus, $G-L$ must be a bipartite graph. 
 
 Let $V(G-L):= V_1\sqcup V_2$ be a bipartition  of vertices of $G-L$. Recall that $d_{G-L}(v)> \frac{1}{2}n-\frac{9}{4}\sqrt{kn}$ for each $v\in V(G-L)$. The neighbors of $v\in V_1$ must lie in another vertex part $V_2$.
 So we have 
 $\frac{1}{2}n-\frac{9}{4}\sqrt{kn}\leq|V_1|,|V_2|\leq \frac{1}{2}n+\frac{9}{4}\sqrt{kn}$, as needed. 
\end{proof}

The next lemma gives $|L|<2k$, which 
refines the bound obtained from Lemma \ref{lem-bound-L}. 

\begin{lemma} \label{lem-L-2k}
   We have $L\subseteq V(C)$ and $|L|< 2k$. 
\end{lemma}
\begin{proof}
    We may assume by scaling that $\max\{x_v:v\in V(G)\}=1$. Let $z\in V(G)$ be a vertex with $x_z=1$. It is easy to see that $z\notin L$. Otherwise, we get $\lambda(G) = \lambda (G)x_z = \sum_{w\in N(z)} x_w \le 
    d(z)\leq \frac{n}{2} -\frac{5}{4}\sqrt{kn}< \frac{n}{2} - k $, contradicting with Lemma \ref{lem-graph-size}. 
    Without loss of generality, we may assume that $z\in V_1$. 
    Then
    \begin{align*}
        \lambda(G)=\lambda(G)x_z =\sum_{v\in N(z)\cap L}x_v+\sum_{v\in N(z)\cap V_2}x_v 
        < |L|+\sum_{v\in V_2}x_v,
    \end{align*}
    which implies  that 
    \begin{equation}\label{equation 1}
        \sum_{v\in V_2}x_v>\lambda(G)-|L|.
    \end{equation}

    Now we are going to prove that $L\setminus V(C)=\emptyset$. Suppose on the contrary that there is a vertex $w\in L\setminus V(C)$. Then $d(w)\leq \frac{n}{2} -\frac{5}{4}\sqrt{kn}$. Let $G':=G-\{wv : v\in N(w)\}+\{wv : v\in V_2\setminus  V(C)\}$. 
    We claim that $G'$ is non-bipartite  and 
    $\mathcal{C}_{\ell,k}$-free. Otherwise, if $G'$ contains a copy of an odd cycle $C_{2b+1}\in \mathcal{C}_{\ell,k}$. From the construction of $G'$, we see that $w\in V(C')$, and $w$ has two neighbors $v_1,v_2\in V_2\setminus  V(C)$. For each $i\in \{1,2\}$, by Lemma \ref{V_1V_2}, we have 
    \[ d_{V_1}(v_i)\ge \left(\frac{n}{2}-\frac{9}{4}\sqrt{kn} \right) -2k.\]
    It follows that  
    $$|N_{V_1}(v_1)\cap N_{V_1}(v_2)| \ge 2\left(\frac{n}{2} - \frac{9}{4}\sqrt{kn} -2k \right) - 
    \left(\frac{n}{2} + \frac{9}{4}\sqrt{kn} \right)>2k.$$ 
    Consequently, we can find a vertex  $w'\in N(v_1)\cap N(v_2)\cap V_1$ such that 
    $w'\neq w$ and 
     $\big( V(C_{2b+1})\setminus \{w\} \big) \cup \{w'\}$ 
     forms a copy of $C_{2b+1}$ in the original graph $G$, a contradiction. 
    
    Since $n\geq 144k$ and $k\geq 2$, we have
\begin{align*}
    \lambda(G')-\lambda(G)&\geq\frac{{\bm{x}}^T \big(A(G')-A(G) \big){\bm{x}}}{{\bm{x}^T}{\bm{x}}} \\ 
    &=\frac{2x_{w}}{{\bm{x}}^T{\bm{x}}}\left( \sum_{u\in V_2\setminus  V(C)}x_u-\sum_{u\in N_G(w)}x_u \right)\\
    &\geq \frac{2x_{w}}{{\bm{x}^T}{\bm{x}}}
    \Big(\lambda(G)-|L|-(2k-1)-d(w) \Big)\\
    &> \frac{2x_{w}}{{\bm{x}^T}{\bm{x}}}
    \left(\frac{1}{4}\sqrt{kn}-3k \right)>0,
\end{align*}
which contradicts with the maximality of $G$. Then $L\subseteq V(C)$ and $|L|<2k$. 
\end{proof}
Recall that $C=u_1\ldots u_{2t+1}u_1$ is a shortest odd cycle in $G$. 
For simplicity, we denote 
\[ V_1':=V_1\setminus V(C) ~~\text{and}~~ 
V_2':=V_2\setminus V(C). \]  
By Lemmas \ref{V_1V_2} and \ref{lem-L-2k}, 
we know that $G-L$ is bipartite and $L \subseteq V(C)$. 
Then $G-V(C)$ is a bipartite graph with partite sets $V_1'$ and $V_2'$.  
We conclude that $V(G)=V_1'\sqcup V_2' \sqcup V(C)$ and   $e(V_1')=e(V_2')=0$. Note that 
$\delta(G-C) \ge \frac{1}{2}n - \frac{5}{4}\sqrt{kn} -2k$. 
Thus, we get 
\begin{equation}
    \label{eq-new-V12}
    \frac{1}{2}n-\frac{5}{4}\sqrt{kn}-2k\leq|V_1'|,|V_2'|\leq \frac{1}{2}n+\frac{5}{4}\sqrt{kn}+2k.
\end{equation}

\begin{lemma}\label{lem-pathlength}
Let $v,v'\in V_1'$ and $u,u'\in V_2'$. 
The following hold: 
\begin{itemize}
    \item[\rm (a)]  For every $h\in\{3,5,\ldots,2k-1\}$ and $s\in\{2,4,\ldots,2k-2\}$, there exist a path from $v$ to $u$ with length $h$ in  $G-C$, and a path from $v$ to $v'$ with length $s$ in $G-C$.
    
    \item[\rm (b)] For each $u_i\in V(C)$, we have $d_{V_1'}(u_i)\cdot d_{V_2'}(u_i)=0$.
\end{itemize}
\end{lemma}

\begin{proof} For part (a), we 
observe that $v,u\notin L$ by \Cref{lem-L-2k}. Then 
    $$d_{V_2'}(v)>\frac{1}{2}n-\frac{5}{4}\sqrt{kn}-2k$$ 
    and 
    $$d_{V_1'}(u)> \frac{1}{2}n-\frac{5}{4}\sqrt{kn}-2k.$$ 
 Next, we can greedily find a path of length $h$ starting from $v$ and ending at $u$. 
     Indeed, we denote $v_1:=v$ and 
     $v_{h+1}:=u$. Since $d_{V_2'}(v_1)>\frac{1}{2}n-\frac{5}{4}\sqrt{kn}-2k$, 
     we can choose a vertex $v_2\in N_{V_2'}(v_1)\setminus\{v_{h+1}\}$. 
     Since $v_2\in V_2'$, we   
     consider the neighbors of $v_2$ in 
     $V_1'$. More precisely, a direct computation gives 
     \begin{align*}
     |N_{V_1'}(v_2)| >\frac{1}{2}n-\frac{5}{4}\sqrt{kn}-2k>k.
     \end{align*} 
     Similarly, we can find  $v_3,v_5, \ldots ,v_{h-2}\in V_1'$ and 
     $v_4,v_6,\ldots ,v_{h-1} \in V_2'$
     such that $v_iv_{i+1}\in E(G-C)$ for each $1\le i\le h-2$. Finally, we show that there exists a vertex $v_{h}\in  V_1'$ such that 
     $v_{h-1}v_h, v_hv_{h+1}\in E(G-C)$. This is available for our purpose, since 
     $d_{V_1'}(v_{h-1}),d_{V_1'}(v_{h+1}) \ge 
     \frac{n}{2} - \frac{5}{4}\sqrt{kn}-2k$. Moreover, we know from (\ref{eq-new-V12}) that $|V_1'|\le \frac{n}{2} + \frac{5}{4}\sqrt{kn}+2k$.  
     Then for $n\ge 85k$, 
     \[ |N_{V_1'}(v_{h-1}) \cap N_{V_1'}(v_{h+1}) | \ge 2\left(\frac{n}{2}-\frac{5}{4}\sqrt{kn}-2k\right)-\left(\frac{n}{2}+\frac{5}{4}\sqrt{kn}+2k\right)-k>0.  \]
Therefore, we can choose a vertex $v_{h}\in V_1'$ that joins both $v_{h-1}$ and $v_{h+1}$. 
     In conclusion, 
     we find a path starting from $v$ to $u$ with length $h$ in $G-C$. 
Similarly, we can prove that for every  $s\in\{2,4,\ldots,2k-2\}$, there exists a path from $v$ to $v'$ with length $s$ in $G-C$.

For part (b), suppose on the contrary that there exists  a vertex $u_i\in V(C)$ such that $d_{V_1'}(u_i)\neq 0 $ and 
$d_{V_2'}(u_i)\neq 0$. Then we can choose two vertices $v_1\in N_{V_1'}(u_i)$ and $v_2\in N_{V_2'}(u_i)$. By part (a), there exists a path from $v_1$ to $v_2$ with length $2k-1$ in $G-C$. Thus, we find an odd cycle $C_{2k+1}$ in $G$, which contradicts with the assumption.
\end{proof}

Let $u_1\in V(C)$ be a vertex such that $x_{u_1}=\max\{x_{u_i} : 1\le i\le 2t+1\}$.

\begin{lemma} \label{lem-deg-zero}
    For each $i\in[2,2t+1]$,
    we have $d_{V(G-C)} (u_i)=0$.  
\end{lemma}

\begin{proof}
Recall that $V(G)=V_1'\sqcup V_2' \cup V(C)$. 
Suppose on the contrary that $d_{G-C}(u_i)\ge 1$. In the sequel, we will deduce a contradiction. 
For each $i\in[2,2t+1]$, we denote by   $Q_1=u_1u_2\ldots u_i$ and $Q_2=u_1u_{2t+1}u_{2t}\ldots u_i$ the two paths from $u_1$ to $u_i$ around the cycle $C$. {The maximality of $x_{u_1}$ implies that  $u_1$ has at least one neighbor in $V(G-C)$.} By \Cref{lem-pathlength}, we 
may assume by symmetry that $u_1$ has neighbors only in $V_2'$.

{\bf Case 1.} Assume that $d_{V_2'}(u_1)\geq 2$. 

In this case, 
 there exist two  vertices  $v_1\neq v_2$ such that $v_1\in N_{V_2'}(u_1)$ and $v_2\in N_{G-C}(u_i)$. 

    {\bf Case 1.1. $t\le k-2$.} 

If $v_2\in V_1'$, then by \Cref{lem-pathlength}, there exist paths in $G-C$ from $v_1$ to $v_2$ with all odd lengths at least $3$. 
Since $C$ is an odd cycle with length at most $ 2k-3$,  
we see that either $v_1Q_1u_iv_2$ or $v_1Q_2u_iv_2$ is a path with even length at most $2k-2$. 
This even path together with an odd path between $v_1$ and $v_2$ can form an odd cycle $C_{2k+1}$, a contradiction.

If $v_2\in V_2'$, then by \Cref{lem-pathlength}, there exist paths in $G-C$  from $v_1$ to $v_2$ with all even lengths at least $2$. Similarly, either 
 $v_1Q_1u_iv_2$ or $v_1Q_2u_iv_2$ 
 is a path with odd length at most $2k-3$. 
 This odd path together with an even path between $v_1$ and $v_2$ can lead to a copy of $C_{2k+1}$. This is also a contradiction.

{\bf Case 1.2. $t=k-1$.}  

The cycle $C$ has length $2k-1$ exactly. 
Similar to the above discussion, 
we can prove that for every $i\in[3,2k-2]$, 
we have $d_{G-C}(u_i)=0$.  
Next, we need to prove $d_{V(G-C)}(u_2)=0$ and $d_{V(G-C)}(u_{2k-1})=0$. 
We remark that the arguments for $u_2$ and $u_{2k-1}$ are different from that for other vertices\footnote{For example, if $u_2$ has a neighbor, say $v_2$, in the part $V_1'$. 
Then $v_1Q_2u_2v_2$ is a path with length $2k$. Although there exist paths from $v_1$ to $v_2$ with all odd lengths at least $3$, this can not lead to a copy of $C_{2k+1}$.}. 
Firstly, we can see that $N_{G-C}(u_2) \subseteq {V_1'}$ and $N_{G-C}(u_{2k-1}) \subseteq {V_1'}$. Otherwise, 
if there is a vertex $v \in N_{V_2'}(u_2)$, 
then by \Cref{lem-pathlength}, we can find a path from $v$ to $v_1$ with length $2k-1$, 
which together with the path $v_1u_1u_2v$ yields a copy of $C_{2k+1}$, a contradiction. 

If $u_2$ has a neighbor $v\in V_1'$ and $u_{2k-1}$ has a neighbor $w\in V_1'$ such that $w\neq v$, then $vu_2u_3\cdots u_{2k-1}w$ forms a path with length $2k-1$. 
Clearly, we have $|N(v) \cap N(w)| \ge 1$.   Extending the above path leads to a copy of $C_{2k+1}$, a contradiction. 
Therefore, at least one of $N_{V_1'}(u_{2})$ and 
$N_{V_1'}(u_{2k-1})$ is empty, 
or $N_{V_1'}(u_2)=N_{V_1'}(u_{2k-1})=\{v\}$ for some vertex $v\in V_1'$. 
In what follows, we only deal with the former case since the latter case is similar. 
Without loss of generality, we may assume that $N_{V_1'}(u_2)\neq \emptyset$ and $N_{V_1'}(u_{2k-1})= \emptyset$. 
In this case, there is no edge between 
$N_{V_1'}(u_2)$ and $N_{V_2'}(u_1)$. 
 Otherwise, we can find a copy of $C_{2k+1}$, a contradiction.      
Let $G'=G-\{u_2v : v\in N_{V_1'}(u_2) \} + 
\{vu : v\in N_{V_1'}(u_2),u\in N_{V_2'}(u_1)\}$. Clearly, $G'$ is a non-bipartite $\mathcal{C}_{\ell,k}$-free graph on $n$ vertices. 
Note that $\lambda (G) x_{u_1} = x_{u_2} + x_{u_{2t+1}} + \sum_{u\in N_{V_2'}(u_2)} x_u$. 
Then 
  \begin{align*}
    \lambda(G')-\lambda(G)
    &\geq\frac{2}{{\bm{x}^T}{\bm{x}}}\left( 
    \sum_{v\in N_{V_1'}(u_2)}\sum_{u\in N_{V_2'}(u_1)} x_vx_u 
    - \sum_{v\in N_{V_1'}(u_{2})}x_{u_2}x_v \right) \\
    &\geq \frac{2}{{\bm{x}^T}{\bm{x}}}\sum_{v\in N_{V_1'}(u_2)} 
    x_v \left(\sum_{u\in N_{V_2'}(u_1)}x_u 
    -x_{u_2} \right)\\
    &= \frac{2}{{\bm{x}^T}{\bm{x}}}\sum_{v\in N_{V_1'}(u_2)} x_v \Big(\lambda(G)x_{u_1}-2x_{u_2} - x_{u_{2k-1}} \Big)\\
    &\geq \frac{2}{{\bm{x}^T}{\bm{x}}}\sum_{v\in N_{V_1'}(u_2)} x_v 
    (\lambda(G)-3) x_{u_1}\\
    &> 0, 
\end{align*}
which contradicts with the maximality of $G$. 
Hence, we must have $d_{G-C}(u_2)=0$. 

{\bf Case 2.} 
Assume that $d_{V_2'}(u_1)=1$. 

Recall that $C$ is a shortest odd cycle in $G$. By the above argument, 
{we can similarly show that  $d_{V(G-C)}(u_{i})=0$ for every $i\in[4,2t-1]$.} 
Moreover, for each $i\in \{2,3,2t,2t+1\}$, we have 
\[ d_{V(G-C)}(u_i)\leq 1. \] 
Let $G'=G-\{u_iw: w\in N_{V(G-C)}(u_i), \ i\in\{2,3,2t,2t+1\}\}+\{u_1w: w\in V_2'\setminus N_{V_2'}(u_1)\}$. Obviously, $G'$ is non-bipartite and $\mathcal{C}_{\ell,k}$-free. 
By (\ref{equation 1}) and \Cref{lem-L-2k}, 
we get 
$$\sum_{u\in V_2'\setminus N_{V_2'}(u_1)} x_u 
\ge \lambda(G) - |L| -1 \ge \lambda(G) -2k.$$ 
Note that $x_w\le x_{u_1} \le 1$ for every $w\in V(G)$. 
It follows that
\begin{align*} 
    \lambda(G')-\lambda(G) &=\frac{2}{{\bm{x}}^T{\bm{x}}}\left( \sum_{u\in V_2' \setminus N_{V_2'}(u_1)} x_{u_1}x_u-\sum_{i\in\{2,3,2t,2t+1\}}\sum_{w\in N_{V(G-C)}(u_i)} x_w x_{u_i} \right)\\ 
    & { \geq\frac{2}{{\bm{x}}^T{\bm{x}}}\left( \sum_{u\in V_2' \setminus N_{V_2'}(u_1)} x_{u_1}x_u-\sum_{i\in\{2,3,2t,2t+1\}}d_{V(G-C)}(u_i) \cdot x_{u_i} \right) } \\
    &\geq \frac{2}{{\bm{x}^T}{\bm{x}}}\left(\sum_{u\in V_2'\setminus N_{V_2'}(u_1)}x_{u_1}x_u- 4 x_{u_1}\right)\\ 
    &\geq \frac{2x_{u_1}}{{\bm{x}^T}{\bm{x}}}
    \Big(\lambda(G)-2k-4 \Big)\\
    &>0, 
\end{align*}
a contradiction. This completes the proof.
\end{proof}

From the previous discussion, we know that $V(G)=V_1'\sqcup V_2'\sqcup V(C)$ such that $e(V_1')=e(V_2')=0$ and $u_i$ has no neighbor in $V(G-C)$ for each $i\in [2,2t+1]$. Since $G$ is the spectral extremal graph, i.e., maximizing the spectral radius, it follows that $G[V_1',V_2']$ forms a complete bipartite subgraph, and $u_1$ is adjacent to all vertices of $V_2'$ by symmetry. 
Putting $u_1$ and $V_1'$ together, we see that 
$G[V_1'\cup \{u_1\},V_2']$ is a complete bipartite graph.

To finish the proof of \Cref{thm-odd-bi-stable}, we need to show that 
$|V_2'| -1 \le |V_1'| +1 \le |V_2'|$. 
 {It is tricky to use spectral techniques in order to prove  such an  inequality.} 
 Our arguments are based on the analysis of the Perron eigenvector, including the classical Rayleigh quotient and 
 the double-eigenvector technique. The latter technique was initially used by Rowlinson \cite{Row1988}; see, e.g., \cite{ZLX2022,LZZ2022} for recent applications. 
So we present this tricky part as the following lemma.

\begin{lemma} \label{lem-balanced}
   We have  $|V_2'| -1 \le |V_1'| +1 \le |V_2'|$. 
\end{lemma}

\begin{proof}
     Recall that $C=u_1u_2\cdots u_{2t+1}u_1$ is a shortest odd cycle, where $x_{u_1}=\max\{x_{u_i}:1\le i\le 2t+1\}$ and $G[V_1'\cup\{u_1\},V_2']$ is a complete bipartite graph. First of all, 
     we claim that 
     \begin{equation}
         \label{eq-max-u1}
         x_{u_1}>\max\{x_v:v\in V(G) \setminus \{u_1\}\} . 
     \end{equation}
      Otherwise, if $x_{u_1} \leq x_v$ holds for a vertex $v\in V(G)$ with  $v\neq u_1$, then we denote $G'=G-\{u_1u_2,u_1u_{2t+1}\}+\{vu_2,vu_{2t+1}\}$, and Lemma \ref{lem-WXH} yields $\lambda (G')>\lambda (G)$, a contradiction. 
     Recall that $\bm{x}\in \mathbb{R}^n$ is  a Perron eigenvector with $\max \{x_i:i\in V(G)\}=1$, which implies  $x_{u_1}=1$. 
      
\begin{claim} \label{claim}
    For each vertex $v\in V_1'\cup V_2'$, we have $x_v>\frac{1}{2}$. 
    \end{claim}
    
\begin{proof}[Proof of Claim \ref{claim}] 
Note that $\lambda (G)x_{u_1} = (\sum_{w\in V_2'} x_w ) + x_{u_2} + x_{u_{2t+1}}$. Then 
    for each $u\in V_1'$, 
    \[  \lambda (G) x_u = 
    \sum_{w\in V_2'} x_w = \lambda(G) x_{u_1} - x_{u_2} - x_{u_{2t+1}} \ge \lambda (G) -2.  \]
    Recall in Lemma \ref{lem-graph-size} that  
    $\lambda (G)> \frac{n}{2} -k$. So we 
    get $x_u\ge 1- \frac{2}{\lambda (G)}>\frac{1}{2}$ for each 
    $u\in V_1'$. 
   
   Similarly, for every $v\in V_2'$, we have  
   \[ \lambda(G) x_v =x_{u_1}+ \sum_{u\in V_1'} x_u \ge (1+|V_1'|) 
   \left(1- \frac{2}{\lambda(G)} \right). \] 
   Consequently, we obtain  
   $x_v\ge \frac{(1+ |V_1'|)(1-\frac{2}{\lambda (G)})}{\lambda (G)}$. Recall in (\ref{eq-new-V12}) that $|V_1'|\ge \frac{n}{2} - \frac{5}{4}\sqrt{kn} -2k$ and $\frac{n}{2} -k < \lambda (G)\le \frac{n}{2}$. So $x_v > \frac{1}{2}$ 
   for every $v\in V_2'$, as needed.  
\end{proof}
     
     Next, we show that $|V_1'| +1\le |V_2'|$. Suppose on the contrary that $|V_1'|\geq |V_2'|$. We denote 
     $s:=|V_1'|-|V_2'|\geq 0$. We will show that for every $v\in V_{2}'$, we have $x_{u_1}< x_v$, which leads to a contradiction with (\ref{eq-max-u1}).  Note that all entries of $\bm{x}$ corresponding to vertices of $V_i'$ are equal. 
     Fix two vertices $u\in V_1'$ and $v\in V_2'$, we have 
     $
     \lambda(G)x_u=\sum_{w\in V_2'}x_w=|V_2'|x_v
     $. Then 
     $$
\lambda(G)x_{u_1}=|V_2'|x_v+x_{u_2}+x_{u_{2t+1}} = \lambda (G)x_u + x_{u_2} + x_{u_{2t+1}}.
     $$
    Consequently, we obtain $x_u = x_{u_1} - \frac{1}{\lambda (G)}(x_{u_2} + x_{u_{2t+1}})$ and 
      \begin{equation} \label{eq-xv}
     \lambda(G)x_v= |V_1'|x_u+x_{u_1}  
     =(|V_1'|+1)x_{u_1}-\frac{|V_1'|}{\lambda(G)}(x_{u_2}+x_{u_{2t+1}}). 
     \end{equation} 
   Note that $x_{u_1}=1$ and $x_{u_i} \le \frac{2}{\lambda (G)}$ for every $2\le i\le 2t+1$. 
    Then 
    $$\sum_{uv\in E(C)} x_ux_v \le 
    \frac{4}{\lambda(G)} + 
    (2t-1)\frac{4}{\lambda^2(G)} < \frac{5}{\lambda (G)}.$$ 
    By Claim \ref{claim}, we have 
    $\bm{x}^T\bm{x} \ge 1+ \frac{1}{4}(n-2t-1) \ge \frac{1}{4}(n-2k)$.  
    Let $G'=G-\{e\in E(G): e\in E(C)\}$. 
    Removing the $2t$ isolated vertices, we see that $G'$ is a complete bipartite graph with vertex parts $V_1'\cup\{x_{u_1}\}$ and $ V_2'$.  
    By the Rayleigh quotient formula, 
it follows that 
    \begin{equation} \label{eq-lam}
\lambda (G)\le \lambda (G') + \frac{2\sum_{uv\in E(C)}x_ux_{v}}{\bm{x}^T\bm{x}} \le \sqrt{(|V_1'|+1)\cdot |V_2'|} + \frac{40}{\lambda(G)(n-2k)}. 
    \end{equation}
Combining (\ref{eq-xv}) with (\ref{eq-lam}), we obtain 
\begin{align*}
    x_v -x_{u_1}&=\frac{|V_1'|+1-\lambda(G)}{\lambda(G)}x_{u_1}-\frac{|V_1'|}{\lambda^2(G)}(x_{u_2}+x_{u_{2t+1}})\\
    &\geq \left(\frac{|V_1'|+1-\sqrt{(|V_1'|+1)\cdot|V_2'|}-\frac{40}{\lambda(G)(n-2k)}}{\lambda(G)}\right)x_{u_1}-\frac{4|V_1'|}{\lambda^3(G)}\\
    &\geq \left(\frac{\sqrt{|V_1'|+1}}{\sqrt{|V_1'|+1}+\sqrt{|V_2'|}}(s+1)-\frac{40}{\lambda (G) (n-2k)}\right)\frac{x_{u_1}}{\lambda(G)}-\frac{4|V_1'|}{\lambda^3(G)}\\
    &\overset{(\ref{eq-new-V12})}{\geq} \frac{1}{4}\cdot  \frac{x_{u_1}}{\lambda(G)}-\frac{4|V_1'|}{\lambda^3(G)} \quad  (\text{since ${x_{u_1}=1}$})\\
    &\geq \frac{\lambda^2(G)- 16|V_1'|}{4\lambda^3(G)}, 
\end{align*}
which leads to $x_v> x_{u_1}$, a contradiction. Thus, we have $|V_1'|+1\leq |V_2'|$. 

Finally, we prove that $|V_2'|-1\leq |V_1'|+1$. Suppose on the contrary that $ |V_1'|+1 \le |V_2'|-2$. We will show that deleting a vertex of $V_2'$ and copying a vertex of $V_1'$  will increase the spectral radius of $G$. 
The key ingredient depends on the double-eigenvector technique. 
To start with, we choose a vertex $v\in V_2'$, and denote $G'=G-\{vw:w\in V_1'\cup \{u_1\}\}+\{vw:w\in V_2'\}$. Let $\bm{y}\in \mathbb{R}^n$ be the Perron eigenvector of $\lambda(G')$ with $\max\{y_v:v\in V(G')\}=1$. Fix two vertices $v_1\in V_1'$ and $v_2\in V_2'\setminus \{v\}$,  
we shall prove the following claim. 

\begin{claim} \label{claim-3-10}
    We have $\lambda(G)y_v(x_{v_1}-x_v) + \lambda(G')x_v(y_v-y_{v_2}) >0$. 
\end{claim}

\begin{proof}[Proof of Claim \ref{claim-3-10}]
We denote $s:=|V_2'|-|V_1'|\geq 3$. Note that 
the eigen-equation gives 
$$
\begin{cases}
\lambda(G)x_{v_1}=|V_2'|x_v, \\ \lambda(G)x_v=|V_1'|x_{v_1}+x_{u_1} . 
\end{cases} $$
Then 
\begin{equation} \label{eq-xxx}
 x_{v_1}-x_v=\frac{sx_v-x_{u_1}}{\lambda(G)+|V_1'|}.
 \end{equation}
Correspondingly, we have 
$$
\begin{cases}
\lambda(G')y_{v}=(|V_2'|-1)y_{v_2}, \\ 
\lambda(G')y_{v_2}=(|V_1'|+1)y_{v}+y_{u_1}, \\
\lambda (G') y_{u_1}= (|V_2|-1)y_{v_2} + y_{u_2} + y_{u_{2t+1}}.
\end{cases}
$$ 
It yields that 
\begin{equation}
    \label{eq-yyy}
    y_{v}-y_{v_2}=\frac{(s-2)y_{v_2}-y_{u_1}}{\lambda(G')+|V_1'|+1}.
\end{equation} 
A similar argument  of Claim \ref{claim} can yield that $x_{v},y_v,x_{v_1},y_{v_2}>\frac{1}{2}$. 

{\bf Case 1.}
If $s\geq 4$, then we obtain from (\ref{eq-xxx}) and (\ref{eq-yyy}) that 
$x_{v_1}-x_v>0$ and $y_{v}-y_{v_2}>0$. 

{\bf Case 2.} 
Suppose that $s=3$. Note that 
$\lambda(G')\le \Delta (G')=|V_1'| +4$. 
We have 
\begin{align}
\lambda (G')(y_{v_2} - y_{u_1} )&= 
(|V_1'| +1)y_v + y_{u_1} - (|V_2'|-1)y_{v_2} - y_{u_2} - y_{u_{2t+1}} \notag \\ 
&=(|V_1'|+1-\lambda(G'))y_v+y_{u_1}-y_{u_2}-y_{u_{2t+1}} \notag \\
& \ge -3y_v+y_{u_1}-y_{u_2}-y_{u_{2t+1}}. \label{eq-zzz} 
\end{align}
Using (\ref{eq-xxx}) and (\ref{eq-yyy}), we get 
\begin{align*}
    &\lambda(G)y_v(x_{v_1}-x_v) + \lambda(G')x_v(y_v-y_{v_2}) \\
    &= 
    \frac{\lambda(G)y_v}{\lambda(G)+|V_1'|}(3x_v-x_{u_1})+\frac{\lambda(G')x_v}{\lambda(G')+|V_1'|+1}(y_{v_2}-y_{u_1})\\
    &\overset{(\ref{eq-zzz})}{\geq}  
    \frac{\lambda(G)y_v}{\lambda(G)+|V_1'|}(3x_v-x_{u_1})+\frac{x_v}{\lambda(G')+|V_1'|+1}\cdot (-3y_v+y_{u_1}-y_{u_2}-y_{u_{2t+1}}) \\ 
    &>\frac{\lambda(G)x_vy_v}{\lambda(G)+|V_1'|}-\frac{5x_v}{\lambda(G')+|V_1'|+1}\\
    &>0.
\end{align*}
To sum up, {we always get $\lambda(G)y_v(x_{v_1}-x_v) + \lambda(G')x_v(y_v-y_{v_2}) >0$ in either cases above.}
\end{proof}

 Applying the double-eigenvector technique, it follows that 
\begin{align*}
    &\bm{y}^T(\lambda(G')-\lambda (G)) \bm{x}=\bm{y}^T(A(G')-A(G)) \bm{x} \\
    &= \sum_{uv\in E(G')} (x_uy_v+x_vy_u) - 
    \sum_{uv\in E(G)} (x_uy_v + x_vy_u) \\ 
    &=
    x_v \left(\sum_{w\in V_2'\setminus\{v\}}y_w \right)+y_v \left(\sum_{w\in V_2'\setminus\{v\}}x_w \right) 
    - x_v\left(\sum_{w\in V_1'\cup\{u_1\}}y_w\right)-y_v\left(\sum_{w\in V_1'\cup\{u_1\}}x_w\right) \\
    &=\lambda(G')x_vy_v+y_v \big(\lambda(G)x_{v_1}-x_v\big)- x_v\big(\lambda(G')y_{v_2}-y_v \big)-\lambda(G)x_vy_v \\
    &= \lambda(G)y_v(x_{v_1}-x_v)+ \lambda(G')x_v(y_v-y_{v_2}) \\
    &>0, 
\end{align*}
where the last inequality by Claim \ref{claim-3-10}, a contradiction. So $|V_2'|-1\leq |V_1'|+1$, as needed.  
\end{proof}

    By the above discussion, we get $G= C_{2t+1}(T_{n-2t,2})$, where $\ell \le t \le k-1$. Then 
    $$\lambda(G)= \lambda(C_{2t+1}(T_{n-2t,2}))\leq \lambda(C_{2\ell+1}(T_{n-2\ell ,2})),$$ 
    where the last inequality holds by \Cref{lem-final}. Moreover, the equality holds if and only if $t=\ell$ and 
    $G= C_{2\ell+1}(T_{n-2\ell ,2})$.  This completes the proof of \Cref{thm-odd-bi-stable}.

\section{Concluding remarks}

In this paper, we have studied the spectral extremal problems for non-bipartite graphs without short odd cycles. However, a distinct phenomenon is observed in non-bipartite triangle-free graphs. 
Recall that a result of Erd\H{o}s \cite[p. 306]{BM2008} states that 
if $G$ is a non-bipartite triangle-free graph on $n$ vertices, then 
$e(G)\le \lfloor \frac{(n-1)^2}{4}\rfloor +1$. 
This bound is the best possible when considering the blow-ups of a $5$-cycle. 
Extending this result, a natural way is to consider the triangle-free graphs with 
restricted chromatic number. In fact, this seems much more difficult. 

In 1995, Jin \cite{Jin1995} proved that if $G$ is an $n$-vertex graph with 
chromatic number $\chi (G)\ge 4$, then the minimum degree $\delta (G)\le \frac{10}{29}n$. This bound can be achieved by a balanced blow-up of the Gr\"{o}tzsch graph.  The Gr\"{o}tzsch graph is defined as in Figure \ref{Fig-Grot}, which is the (order-)smallest triangle-free graph with chromatic number four.

 \begin{figure}[H]
 \centering  
\includegraphics[scale=0.8]{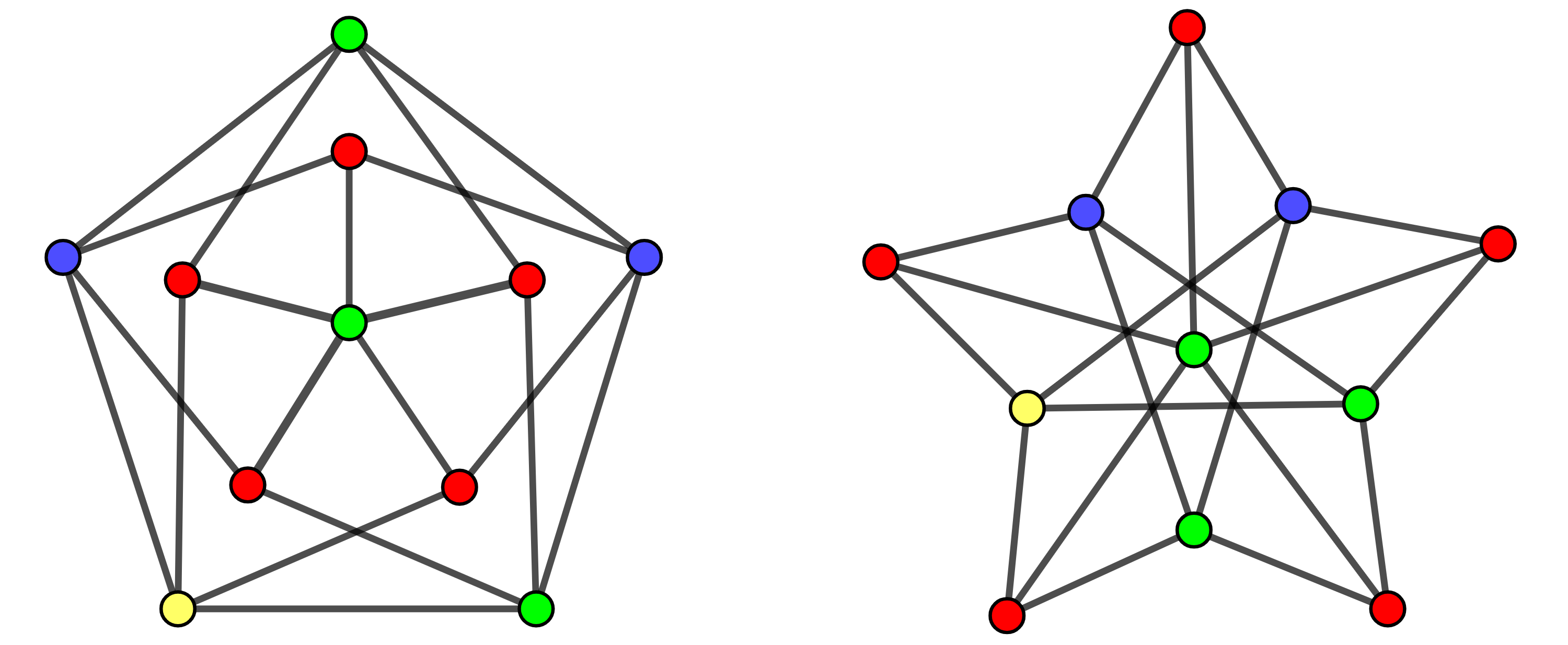}  
\caption{Two drawings of Gr\"{o}tzsch's graph} \label{Fig-Grot} 
\end{figure} 

Recently, Ren, Wang, Wang and Yang \cite{RWWY2024} proved that 
if $G$ is a triangle-free graph on $n\ge 150$ vertices with $\chi (G)\ge 4$, then 
\[  e(G)\le \left\lfloor \frac{(n-3)^2}{4} \right\rfloor +5 , \]
with equality if and only if $G$ is a specific unbalanced blow-up of Gr\"{o}tzsch's graph. 

\medskip 
 Motivated by the above results, 
the following problem naturally arises: 

\begin{problem}
 Among all $n$-vertex triangle-free graphs with chromatic number at least four, which graph achieves the maximum spectral radius? 
\end{problem} 

Intuitively, it is highly likely that the spectral extremal graph is also a specific blow-up of Gr\"{o}tzsch's graph. 
It seems possible that the method in our paper could be applied since one can remove few vertices from the extremal graph to make it   bipartite. 
In addition, we note that the smallest triangle-free graph with chromatic number five can be found in \cite{JR1995}. On the whole, we believe that the spectral extremal problem for triangle-free graph with chromatic number five is also a blow-up of that graph in \cite{JR1995}. We would like to leave this problem to interested readers.

More generally, it is meaningful to extend the aforementioned results by forbidding a general clique. 
Let $r\ge 2$ and $t\ge r+1$.  
We now introduce some known results for a 
 $K_{r+1}$-free graph $G$ with $\chi (G)\ge t$. The case $t=r+1$ has been   extensively studied in the literature. 
In 1974, Andr\'{a}sfai, Erd\H{o}s and S\'{o}s \cite{AES1974} proved that 
    if $r\ge 2$ and $G$ is a $K_{r+1}$-free graph on $n$ vertices with $\chi (G)\ge r+1$, then $\delta (G)\le \frac{3r-4}{3r-1}n$.  Moreover, the bound is the best possible. 
    An alternative proof was provided by Brandt \cite{Bra2003}. 
In 1981, Brouwer \cite{Bro1981} showed that 
if $G$ is a $K_{r+1}$-free graph on 
$n\ge 2r+1$ vertices 
with $\chi (G)\ge r+1$, then 
$ e(G)\le e(T_{n,r}) - \left\lfloor \frac{n}{r} \right\rfloor +1$. 
There are many extremal graphs that attain this bound.  
This result was independently studied in many references, see, e.g., \cite{AFGS2013,KP2005}.
The spectral version was recently established by the third author and Peng \cite{LP2022second}. 
Assume that $V_1,V_2,\ldots ,V_r$ are $r$ vertex parts of $T_{n,r}$, where $|V_1|\le |V_2|\le \cdots \le |V_r|$. Choose $v_0\in V_1$ and 
$u,w\in V_r$. Let $Y_{n,r}$ be the graph obtained from $T_{n,r}$ by adding an edge $uw$ and removing an edge subset $\{wv_0\}\cup \{uv:v\in V_1\setminus \{v_0\}\}$. It was proved in \cite{LP2022second} that 
    if $G$ is an $n$-vertex $K_{r+1}$-free graph with chromatic number $\chi (G)\ge r+1$, then 
    $\lambda (G)\le \lambda (Y_{n,r})$, 
   with equality if and only if $G=  Y_{n,r}$. 
Apart from the above results, 
there is no substantial progress on the spectral extremal problem  for $K_{r+1}$-free graphs with $\chi (G) \ge r+2$.

\section*{Acknowledgements}
The authors would like to thank Xiaoli Yuan and 
Yuejian Peng for sharing the reference \cite{YP2023} before its publication. The authors also thank Xiaocong He and Loujun Yu for reading an earlier draft of this paper. 
 Last but not least, the authors thank the anonymous referees for their careful review and for several valuable suggestions which improved the presentation of our manuscript.
Lihua Feng was supported by the NSFC grant (Nos. 12271527 and 12471022) and the NSF of Qinghai Province (No. 2025-ZJ-902T).
Yongtao Li was supported by the Postdoctoral Fellowship Program of CPSF (No. GZC20233196).

\appendix 

\section{Preliminaries for Theorem \ref{thm-non-bi-cycles}}

The {\it maximum average degree} of a graph $G$  is defined as
$
\mad (G) := \max_{S\subseteq V(G)} {2e(S)}/{|S|}.$
Clearly, we have $\mad(G)\leq \lambda (G)$. A subset $B \subseteq V(G)$ is called {\it critical} if the 
average degree of $G[B]$ is equal to $\mad(G)$. 
The following result was recently proved by Zhang \cite{Zhang-wq-2024} .

\begin{lemma}[See \cite{Zhang-wq-2024}]\label{thm-Zhang}
    For two integers $k \geq1$ and $n\geq 2k + 1$, if $G$ is an $n$-vertex graph with  maximum average degree $\mad(G)\leq 2k$,  then $\lambda(G)<\frac{k-1}{2}+\sqrt{kn+\frac{(k+1)^{2}}{4}}.$
\end{lemma}

The following result is a direct consequence of 
\cite[Theorem 3.2]{Sun-Das}. 

\begin{lemma}[See \cite{Sun-Das,LN2023}] \label{lem-Sun-Das-deletion}
    Let $G$ be a graph. For every $v\in V(G)$, we have 
    \[ \lambda^{2}(G-v)\geq\lambda^{2}(G)-2d(v).\] 
\end{lemma}

The following estimate is straightforward. 

\begin{lemma}
    \label{lem-low-up-bound}
    We have 
    $  \lambda (C_3(T_{n-k-2,2})) <  
    \frac{n-k-2}{2} + \frac{32}{(n-k-3)^2} + \frac{16}{n-k-3}$. 
\end{lemma}

\begin{proof}
    Let $H=C_3(T_{n-k-2,2})$ and $V(H)=\{v_1,v_2,\ldots,v_{n-k}\}$, where $v_1,v_2,v_3$ form a triangle of $H$, and $d(v_1)=d(v_2)=2$ and $d(v_3)= \lceil \frac{n-k-2}{2} \rceil +2$. Let $\bm{x}$ be the Perron eigenvector of  $\lambda(H)$ with the maximum coordinate $\max_{i\ge 1} \{x_i\} = x_{v_3}=1$. 
    Let $H'$ be the bipartite graph obtained by removing the triangle of $H$. 
    Then 
    \[ \lambda(H')\le \frac{n-k-2}{2}. \] 
    By the Rayleigh quotient, we get  
    \[ \lambda(H) \le \lambda (H') + \frac{2(x_{v_1}x_{v_2}+x_{v_2}x_{v_3}+x_{v_1}x_{v_3})}{\bm{x}^T\bm{x}}. \]  
    Since $\lambda (H) x_{v_1}=x_{v_2}+ x_{v_3}\le 2$ and 
    $\lambda(H)> \lambda (H') \ge \lfloor \frac{n-k-2}{2} \rfloor$, we get  $x_{v_1}=x_{v_2}< \frac{4}{n-k-3}$. Therefore, we have 
    $\lambda (H)< \frac{n-k-2}{2} + \frac{32}{(n-k-3)^2}  + \frac{16}{n-k-3}$, as needed. 
\end{proof}

{Let $ec(G)$ and 
$oc(G)$ be the lengths of a longest even cycle and longest odd cycle in $G$, respectively.} The following two results are key ingredients  in our proof of \Cref{thm-non-bi-cycles}. 

\begin{theorem}[See \cite{Voss}]\label{thm-Voss-Zuluaga}
(1) Every 2-connected graph $G$ with $\delta(G)\geq k\geq3$ having at least $2k+1$ vertices satisfies
that $ec( G) \geq 2k.$
(2) Every 2-connected non-bipartite graph $G$ with $\delta(G)\geq k\geq3$ having at least $2k+1$
vertices satisfies that oc$(G)\geq2k-1.$ 
\end{theorem}

\begin{theorem}[See \cite{Gould-Haxell-Scott}] 
\label{thm-Gould-Haxell-Scott}
    For every $c>0$, there is a constant $K:=K(c)$ such that the following holds. Let $G$ be a graph on $n\geq\frac{45{K}}{c^4}$ vertices with $\delta (G)\ge cn$. Then $G$ contains a cycle of length $t$ for every even $t\in[4,ec(G)-K]$ and every odd $t\in[K,oc(G)-K].$
\end{theorem}

\section{Proof of Theorem \ref{thm-non-bi-cycles}}    

\label{App-B}

Our strategy proceeds in the following three steps: 
\begin{enumerate}
    \item 
    Using a technique of Zhang \cite{Zhang-wq-2024}, 
    we can prove by Lemma \ref{thm-Zhang} that 
    $\mad (G)> 2(\frac{1}{6} - \varepsilon_1)n$ and 
    $G$ contains an induced subgraph $G_0$ with average degree $d(G_0)\ge \mad (G)$. 

    \item 
    Following the classical method of Li and Ning \cite{LN2023}, we show that $G_0$ contains an induced subgraph $H$ with  minimum degree $\delta (H)\ge \frac{1}{2}d(G_0)$ and average degree $d(H)\ge d(G_0)$. 
    Moreover, we can prove that 
    $\lambda (H)> \frac{1}{4}(n^2-\frac{4}{3}kn -4n +3)$, where $k:=|V(G)\setminus V(H)|$. 

    \item 
    If $H$ is $2$-connected, then applying 
    Theorems \ref{thm-Voss-Zuluaga} and 
    \ref{thm-Gould-Haxell-Scott}, we are done. 
    If $H$ is not $2$-connected, then 
    deleting few cut-vertices from $H$ iteratively, we can reach a subgraph $F$ whose components are $2$-connected. 
    Then Theorems \ref{thm-Voss-Zuluaga} and 
    \ref{thm-Gould-Haxell-Scott} could be applied. 
\end{enumerate}
 
 Let $c=\frac{1}{7}$ and $K=K(\frac{1}{7})$ be given as in Theorem \ref{thm-non-bi-cycles}. Let $0<\varepsilon<\frac{1}{3}$ and  $\varepsilon_1=\varepsilon/4$ be fixed. Let $N(\varepsilon)$ be sufficiently large such that all the inequalities appeared in the following hold when $n\geq N(\varepsilon)$. For convenience, we may  assume that $n$ is sufficiently large. 

Firstly, we prove that $\mad (G)> 2(\frac16-\varepsilon_1)n.$
Let $k$ be the smallest positive integer such that $k\geq \frac{1}{2} \mad(G)$.  
Since $\lambda(G)\geq \lambda(C_3(T_{n-2,2}))> \sqrt{\lfloor\frac{(n-2)^2}{4} \rfloor}$,  \Cref{thm-Zhang} yields 
$$
\sqrt{\frac{(n-2)^2-1}{4}} <  
\lambda(G)<\frac{k-1}{2}+\sqrt{kn+\frac{(k+1)^2}{4}}.$$
By simplifying, we have 
$$
k>\frac{\frac{(n-2)^2-1}{4}+\sqrt{\frac{(n-2)^2-1}{4}}}{n+1+\sqrt{\frac{(n-2)^2-1}{4}}} 
> \left(\frac{1}{6} - \frac{\varepsilon_1}{2}\right)n.
$$
By the definition of $k$, we get $\frac{1}{2}\mad(G)>\lfloor (\frac{1}{6}- \frac{\varepsilon_1}{2})n \rfloor>
(\frac{1}{6}-\varepsilon_1)n$.

Among all critical subsets of $V(G)$, we choose $B$ as a subset with maximum cardinality such that $2e(B) \geq2 (\frac 16-\varepsilon_1)n|B|$. In the sequel, we denote $G_{0}=G[B]$ for simplicity. 

If $G_0\neq G$, then we denote $V(G) \setminus V(G_0)=\left\{u_1,u_2,...,u_\ell\right\}$, where $\ell :=|V(G)\setminus V(G_0)|\geq1.$ For each $ i\leq\ell$, let $G_i$ be the subgraph of $G$ induced by $V(G_0)\cup\{u_{1},u_{2},...,u_{i}\}.$ If
$\sum_{i=1}^{\ell} d_{G_i}(u_i)>(\frac{1}{6}-\varepsilon_1)n\ell$, 
then
$2e(G)>2e(G_0)+2(\frac{1}{6}-\varepsilon_1)n\ell\ge 2(\frac{1}{6}-\varepsilon_1)n|V(G)|$. 
The maximality of $G_0$ yields $G_0=G$, contradicting with the assumption. Thus, we have 
\begin{equation}
    \label{eq-ell}
    \sum_{i=1}^{\ell} d_{G_i}(u_i)\leq \left(\frac{1}{6}-\varepsilon_1 \right)n\ell . 
\end{equation}
If $G_0=G$, then the inequality (\ref{eq-ell}) holds trivially by setting $\ell=0$ and
$\sum_{i=1}^{\ell} d_{G_{i}}(u_{i})=0$. 
In the remaining of our proof, we always admit the inequality (\ref{eq-ell}).

Let $H$ be a subgraph of $G_0$ defined by a sequence of graphs $H_0,H_1,\ldots,H_s$ such that:
\begin{enumerate}
    \item $H_0=G_0$, $H=H_s$;
    \item for every $i\le s-1$, there is $v_i\in V(H_i)$ such that $d_{H_i}(v_i)< (\frac{1}{6}-\varepsilon_1)n$ and $H_{i+1}=H_{i}-v_i$;
    \item  for every $v\in V(H_s)$, we have $ d_{H_s}(v)\geq(\frac{1}{6}-\varepsilon_1)n$. 
\end{enumerate} 
Obviously, we have  $\delta(H)\geq(\frac{1}{6}-\varepsilon_1)n$ and $d(H) \ge d(G_0)\geq 2(\frac{1}{6}-\varepsilon_1)n$. Moreover, 
\begin{equation}
    \label{eq-s}
    \sum_{i=0}^{s-1} d_{H_i}(v_i) \le s\left(\frac{1}{6} - \varepsilon_1 \right)n.
\end{equation}  
    We denote $k:=\ell +s$. 
     Using Lemma \ref{lem-Sun-Das-deletion}, together with (\ref{eq-ell}) and (\ref{eq-s}) yields 
    $$
\frac{(n-2)^2-1}{4}\leq\lambda^2(G)\leq\lambda^2(H) 
+2\sum_{i=1}^{\ell} d_{G_i}(u_i) + 2\sum_{i=0}^{s-1} d_{H_i}(v_i)<\lambda^2(H)+\frac{kn}{3}. 
    $$
 This implies that 
    \begin{equation}
        \label{eq-H-satisfy}
         \lambda^2(H)> \frac{n^2-(4kn)/3-4n+3}{4}>\lambda^2(C_3(T_{n-k-2},2)), 
    \end{equation}
    where the last inequality hods by \Cref{lem-low-up-bound}. 

    \fbox{{\bf Finding even cycles.}} 
 
{\bf Case A. $\bm{H}$ is $\bm{2}$-connected.} 

    Note that $|H|> d(H)\ge (\frac13-2\varepsilon_2)n $. 
    Then  \Cref{thm-Voss-Zuluaga} gives $ec\left(H\right)\geq\lfloor(\frac 13-2\varepsilon_1)n\rfloor\geq (\frac{1}{3}-2\varepsilon_1)n-1.$ Recall that $\delta(H)\geq(\frac{1}{6}-\varepsilon_1)n\geq \frac{1}{7}|H|$. By Theorem \ref{thm-Gould-Haxell-Scott}, $H$ contains all even cycles $C_{\ell}$ with $\ell\in[4,ec(H)-K]$ if $|H|\geq 45\cdot7^4\cdot K$. 
    Let $n_1$ be an integer satisfying: 
\[ \text{(i) $(\tfrac{1}{3}-2\varepsilon_1)n_1-1\geq 45\cdot7^4\cdot K$; \quad ({ii})  $2\varepsilon_1 n_1\geq K+1$.} \]
Then for $n\geq n_1$, we see that $G$ contains all even cycles $C_{\ell}$ for $\ell\in[4,(\frac{1}{3}-\varepsilon)n]$. 

{\bf Case B. $\bm{H}$ is not $\bm{2}$-connected.} 

In this case, we can delete few cut-vertices from $H$ such that all its components are $2$-connected. 
    Let $F$ be a subgraph of $H$ defined by a sequence of graphs $F_0,F_1,\ldots,F_b$ such that:
    \begin{enumerate}
        \item $H=F_0$ and $F=F_b$;
        \item   for every $i\leq b-1$, there is a cut-vertex $u_i$ of $F_i$ and $F_{i+1}=F_i-u_i$;
        \item $F_b$ has no cut-vertex.
    \end{enumerate} 
    Observe that there are at most $6$ components in $H$.  Otherwise, by averaging, there exists a vertex (in a  minimal component of $H$) that has degree at most $\frac{n}{7}$, contradicting with 
    $\delta (H)\ge (\frac{1}{6}- \varepsilon_1)n$. Hence, $H$ has at most $6$ components.  
        Deleting a cut-vertex in each step increases at least one component of a graph.  
    Similarly, there are at most $6$ 
 components in the terminated graph $F$, 
 and every component of $F$ is $2$-connected. Indeed, once we have deleted $6$ vertices from $H$ iteratively, we can reach at least $7$ components in $F_6$, so there exists a vertex with degree at most $\frac{n}{7}$, contradicting with $\delta (F_6)\ge \delta (H)-6$. 
   Thus, $F$ has at most $6$ components as well. Consequently, we have $b=|V(H) \setminus V(F)|\le 5.$ 
 Recall that $d(H)\geq 2(\frac{1}{6}-2\varepsilon_1)n$. We denote 
 $h:=|H|$. 
   Note that 
    $$
    2e(F)\geq 2e(H)-2\sum_{i=0}^{b-1} d_H(u_i)>  2\left(\frac{1}{6}-\varepsilon_1 \right)n  h-10h.
    $$
The Rayleigh formula yields $ \lambda (F)\ge d(F) = 
\frac{2e(F)}{|F|}> \frac{2e(F)}{|H|}> (\frac{1}{3}-2\varepsilon_1)n-10$. Then $|F|\geq d(F)>(\frac{1}{3}-2\varepsilon_1)n-10$. By \Cref{thm-Voss-Zuluaga}, we have $ec(F)\geq \min\{2\delta(F)-1,|F|\}\geq \min\{(\frac{1}{3}-2\varepsilon_1)n-11,|F|\}=(\frac{1}{3}-2\varepsilon_1)n-11$. By Theorem \ref{thm-Gould-Haxell-Scott}, $F$ contains all even cycles $C_{\ell}$ with $\ell\in[4,ec(F)-K]$ if $|V(F)|\geq 45\cdot7^4\cdot K$. 
    Let $n_2$ be an integer satisfying: 
\[ \text{(i) $(\tfrac{1}{3}-2\varepsilon_1)n_2-11\geq 45\cdot7^4\cdot K$; \quad ({ii})  $2\varepsilon_1 n_2\geq K+11$.} \]
Then for $n\geq n_2$, we see that $G$ contains all even cycles $C_{\ell}$ for $\ell\in[4,(\frac{1}{3}-\varepsilon)n]$.

\fbox{{ \bf Finding odd cycles.}} 
    
{{\bf Case 1. $\bm{H}$ is not $\bm{2}$-connected.}} 

Similar to the argument in previous case, 
let $F$ be a graph obtained from $H$ by deleting cut-vertices iteratively.  
    Let $Q_1$ be a component of $F$ with $\lambda(Q_1)=\lambda(F).$     
    Then $\delta(Q_1)\geq \delta (H) -5\ge (\frac13-2\varepsilon)n-5$ and $\lambda(Q_1)>(\frac13-2\varepsilon_1)n-10.$ Since $\lambda (Q_1)\le |Q_1|-1$, we have $|Q_1|>(\frac13-2\varepsilon_1)n-9.$ 
    In what follows, we prove that $Q_1$ is non-bipartite.  Let $Q_2$ be another component of $F$. Since $\delta(F)\geq(\frac 16-\varepsilon_1)n-5$, we have $| Q_2| \geq (\frac{1}{6}-\varepsilon_1)n-5$. This implies that $|Q_1|\leq h-5-(\frac{1}{6}-\varepsilon_1)n+5<h-\frac{1}{6}n+\varepsilon_1n.$ 
    If $(\frac{1}{3}-2\varepsilon_1)n-10>\frac{1}{2}(h-\frac{n}{6}+\varepsilon_1n)$, then $\lambda(Q_1)>\frac{1}{2}|Q_1|$. 
    It follows that $Q_1$ is non-bipartite. Next, we consider the case $(\frac{1}{3}-2\varepsilon_1)n-10\leq \frac{1}{2}(h-\frac{n}{6}+\varepsilon_1n)$, that is, $h\geq \frac 56n-5\varepsilon_1n- 20.$ By Lemma \ref{lem-Sun-Das-deletion} and (\ref{eq-H-satisfy}),  it is easy to see that
    \begin{align*}
\lambda^2(Q_1)=\lambda^2(F)\geq\lambda^2(H)-2\sum_{i=0}^{b-1}d_{F_i}(u_i) \geq \frac{n^2-(4kn/3)-4n+3}{4}-10h >\frac{(h-\frac{1}{6}n+\varepsilon_1n)^2}{4}
    \end{align*}
    Thus, $Q_1$ is non-bipartite.

    By Theorem \ref{thm-Voss-Zuluaga}, we have $oc(Q_1)\geq\min\{2\delta(Q_1)-1,|Q_1|\}\geq\min\{(\frac 13-2\varepsilon_1)-11,\lambda(F_1)\}=(\frac 13-2\varepsilon_1)n-12.$ By Theorem \ref{thm-Gould-Haxell-Scott}, $Q_1$ contains all odd cycles $C_{\ell}$ for $\ell\in[K,(\frac{1}{3}-2\varepsilon_1)n-12-K]$ if $|Q_1|\geq45\cdot7^4\cdot K$. 
    A result of Zhai and Lin \cite{ZL2022jgt} states that every graph $G$ of order $n$ with  $\lambda(G)>\sqrt{\lfloor\frac{n^2}4\rfloor}$ has a cycle  $C_{\ell}$ for every $\ell\in[3,\frac{n}{7}].$ Let $n_3$ be an integer such that
$$({\rm i})\ (\tfrac{1}{3}-2\varepsilon_1)n_3-5\geq 45\cdot7^4\cdot K;\quad  ({\rm ii})\  n_3\geq 320K;\quad  
({\rm iii})\ 2\varepsilon_1 n_3\geq K+12.$$
If $n\geq n_3$, then $G$ contains all odd cycles $C_{\ell}$ for $\ell\in[5,(\frac{1}{3}-\varepsilon)n]$.

 {{\bf Case 2. $\bm{H}$ is $\bm{2}$-connected.}}  
  
 {\bf Subcase 2.1.} If $H$ is non-bipartite, then  Theorem \ref{thm-Voss-Zuluaga} implies $oc(H)\geq (\frac{1}{3}-2\varepsilon_1)n-3$. By Theorem \ref{thm-Gould-Haxell-Scott}, $H$ contains all odd cycles $C_{\ell}$ for $\ell\in[K,(\frac{1}{3}-2\varepsilon_1)n-3-K]$ if $|H|\geq45\cdot7^4\cdot K$. In view of (\ref{eq-H-satisfy}), 
  \Cref{cor-conse-odd-cycles} yields that $H$ contains all odd cycles $C_{\ell}$ for $\ell\in[5,\frac{n}{187}]$. Let $n_4$ be an integer such that
$$({\rm i})\ (\tfrac{1}{3}-2\varepsilon_1)n_4+1\geq 45\cdot7^4\cdot K;\  ({\rm ii})\ n_4\geq 187K;\  
({\rm iii})\ 2\varepsilon_1 n_4\geq K+3.$$
If $n\geq n_3$, then $G$ contains all odd cycles $C_{\ell}$ for $\ell\in[5,(\frac{1}{3}-\varepsilon)n]$, as needed. 

We point out here that 
the difficulty lies in finding odd cycles with consecutive lengths whenever $H$ is a bipartite graph. Our strategy is to show that $G$ contains an odd cycle $C$ that intersects the bipartite graph $H$ in at least two vertices. Since $H$ is dense, 
we can extend the cycle $C$ to a long cycle $C_{\ell}$ for every odd $\ell \in [5,(\frac{1}{3}-\varepsilon)n]$. 
The details are stated as below. 

{\bf Subcase 2.2.}  If $H$ is bipartite, then we first show that 
$|V(G)\setminus V(H)|\le 6$.
Recall that $H$ is a subgraph of $G$ with $|H|=n-k$ and $\lambda^2(H)\geq \frac{n^2-(4kn)/3-4n+3}{4}$. It is easy to see that $ k\leq 6$ when $n\geq 69$ (otherwise, $\lambda(H)^2\geq \frac{n^2-(4kn)/3-4n+3}{4}>\frac{(n-k)^2}{4}$, $H$ is non-bipartite, a contradiction). By Lemma \ref{lem-Nosal}, we have $e(H)\geq \lambda^2(H)\geq \frac{n^2-(4kn)/3-4n+3}{4}$. Let $V(H)=V_1\sqcup V_2$ be the partition.
  Then $\frac{n-k}{2}-\sqrt{n}\leq 
  |V_1|,|V_2|\leq\frac{n-k}{2}+\sqrt{n}$ (otherwise, $e(H)\leq |V_1|\cdot|V_2|<\frac{n^2-(4kn)/3-4n+3}{4}$, a contradiction). Let $L=\{v\in V(H):d_H(v)\leq \frac{n-k}
  {2}-\frac{5}{4}\sqrt{n}\}$. We claim that $|L|\leq \sqrt{n}$. Otherwise, there exists a subset $S$ of $L$ such that $|S|=\lfloor\sqrt{n}\rfloor$. Then, we have
    \begin{align*}
e(H - S) \geq  \frac{n^2-(4kn)/3-4n+3}{4} - \sqrt{n}\left(\frac{n-k}{2}-\frac{5}{4}\sqrt{n} \right) 
>\frac{(n-k-\lfloor\sqrt{n}\rfloor)^2}{4}.
\end{align*}
Mantel's theorem implies that $H$ has a triangle. This is impossible since $H$ is bipartite.

\begin{claim}\label{clm-cycle-length}
      $G$ contains all odd cycles $C_{\ell}$ with $\ell\in[5,(\frac{1}{3}-\varepsilon)n]$.
\end{claim}

\begin{proof}
We first show that $G$ contains an odd cycle $C$ such that $|V(C) \cap V(H)|\ge 3$.  
Indeed, 
\Cref{cor-conse-odd-cycles} implies $G$ contains all odd cycles $C_{\ell}$ with $\ell\in[5,\frac{n}{187}]$. For $n\geq 1683$, we see that $G$ has a copy of $C_9$. Recall that $k=|V(G)\setminus V(H)|\leq 6$. Then we have $|V(C)\cap V(H)|\geq 3$. 

Let $C=v_1v_2\ldots v_{\ell}v_1$ be an odd cycle with $|V(C)\cap V(H)|\geq 2$. Since $H$ is bipartite, we may assume that $v_1\in V(H)$ and $v_2\in V(G)\setminus V(H)$. Denote $v_t$ be the minimum element of $V(C)$ such that $v_2,\ldots,v_{t-1}\in V(G)\setminus V(H)$ and $v_t\in H$. Since $|V(G)\setminus V(H)|\leq 6$, we have $t\leq 8$.

If $t$ is odd, then the path $v_1v_2\cdots v_t$ has even length. Expanding to the odd cycle $C$, the vertices $v_1,v_t$ must be located in different vertex parts of $H$. If $t$ is even, then $v_1,v_t$ are located in the same vertex part of $H$. 
Next, we only illustrate the case that $t$ is even and $v_1,v_t\in V_1$, since the argument in another case is similar. 
For every odd integer $p\in[11,(\frac{1}{3}-\varepsilon)n]$, we find greedily a path of length $p-t+1$ in $H$ starting from $v_1$ and ending at $v_t$. This leads to a copy of $C_{p}$. Indeed, we denote $v_1:=u_1$ and $v_t:=u_{p-t+2}$. 
By computation, we have 
     $|N_{V_2}(u_1)\setminus L|\ge 
     \delta (H) - |L| \ge 
   (\frac{1}{6}-\varepsilon_1)n -\sqrt{n}>0. $ 
     Then we can choose a vertex $u_2\in N_{V_2}(u_1)\setminus L$. 
     Since $u_2\notin L$, we have $d_{V_1}(u_2)\ge \frac{n-k}{2} - \frac{5}{4}\sqrt{n}$. Then  
     \begin{align*}
     |N_{V_1}(u_2)\setminus  L| 
     \ge 
    d_{V_1}(u_2)- |L| \ge 
     \frac{1}{2}(n-k) - \frac{5}{4}\sqrt{n}-\sqrt{n}
     \ge \left(\frac{1}{6}-\frac{\varepsilon}{2} \right)n.
     \end{align*} 
     Similarly, we can find  $u_3,u_5,\ldots ,u_{p-t}\in V_1\setminus L$ and 
     $u_4,u_6,\ldots ,u_{p-t-1} \in V_2\setminus L$
     such that $u_iu_{i+1}\in E(G)$ for each $1\le i\le p-t-1$. Finally, we show that we can choose a vertex $u_{p-t+1}\in  V_2$ such that 
     $u_{p-t+1}\in N(u_{p-t}) \cap N(u_{p-t+2})$. Note that   
     $d_{V_2}(u_{p-t+2})\geq \delta (H)\ge  (\frac{1}{6}-\varepsilon_1)n $ and $d_{V_2}(u_{p-t}) \ge 
     \frac{n-k}{2} - \frac{5}{4}\sqrt{n}$.   Lemma \ref{lem-graph-size} gives  $|V_2|\le \frac{n-k}{2} + \sqrt{n}$.  
     Then for $n> \frac{81} {\varepsilon^2}$,  
     \begin{align*}
         &|N_{V_2}(u_{p-t+2}) \cap N_{V_2}(u_{p-t}) \setminus\{u_2,u_4,\ldots,u_{p-t-1}\}|\\
         &\ge \left(\frac{n-k}{2}-\frac{5}{4}\sqrt{n} \right) 
         +\left(\frac{1}{6}-\varepsilon_1\right)n-\left(\frac{n-k}{2}+\sqrt{n}\right)- \left(\frac{1}{6}-2\varepsilon_1 \right)n\\
         &=\varepsilon_1 n-\frac{9}{4}\sqrt{n}
         >0.
     \end{align*}   
Therefore, we can choose a vertex $u_{p-t+1}$ which joins both $u_{p-t+2}$ and $u_{p-t}$. 
     In conclusion,  
     we can find an odd cycle
     $v_1v_2\cdots v_tu_{p-t+1}\cdots u_2v_1$. It implies that, $G$ contains all odd cycle $C_{\ell}$ for $\ell\in[11,(\frac{1}{3}-\varepsilon)n]$. Combining with \Cref{cor-conse-odd-cycles}, we complete the proof of Claim \ref{clm-cycle-length}.
\end{proof}

By \Cref{clm-cycle-length}, we know that  $G$ contains all odd cycles $C_{\ell}$ for $\ell\in[5,(\frac{1}{3}-\varepsilon)n]$ when $n\geq \max\{1683,\frac{81}{\varepsilon^2}\}$. 
In conclusion, for any $0<\varepsilon\leq \frac{1}{3}$, we just need $n> \max\{n_1,n_2,n_3,n_4,\frac{81}{\varepsilon^2},1683\}$, then $G$ contains all cycles $C_{\ell}$ for $\ell\in[4,(\frac{1}{3}-\varepsilon)n]$. This completes the proof of \Cref{thm-non-bi-cycles}.

\end{document}